\documentclass[a4paper,12pt]{amsart}

\usepackage{array}

\usepackage{epsf}
\newcommand{\inspic}[1]{\begin{tabular}{c}\epsfbox{#1}\end{tabular}}

\newtheorem{theorem}{Theorem}
\newtheorem{corollary}{Corollary}
\newtheorem{lemma}{Lemma}
\newtheorem{proposition}{Proposition}

\def\oM{{\overline{\mathcal{M}}}}
\def\C{\mathbb{C}}
\def\Z{\mathbb{Z}}

\def\F{\mathcal{F}}

\title{Tautological relations in Hodge field theory}

\author{A.~Losev}

\address{Institute for Theoretical and Experimental Physics, Bolshaya 
Che\-remushkinskaya 25, Moscow, 117218, Russia.}
\email{losev@itep.ru}

\author{S.~Shadrin}

\address{Department of Mathematics, University of Zurich,
Win\-ter\-thu\-rer\-stras\-se 190, CH-8057 Zurich, Switzerland.}

\email{sergey.shadrin@math.unizh.ch}

\author{I.~Shneiberg}

\address{Department of Algebra, Faculty of Mechanics and Mathematics,
Moscow State University, Leninskie Gory, GSP, Moscow, 119899, Russia.}

\email{shneiberg@mtu-net.ru}

\begin{document}

\begin{abstract}
We propose a Hodge field theory construction that captures algebraic 
properties of the reduction of Zwiebach invariants to Gromov-Witten invariants. 
It generalizes the Barannikov-Kon\-tse\-vich construction to the case of higher genera 
correlators with gravitational descendants. 

We prove the main theorem stating that algebraically defined 
Hodge field theory correlators
satisfy all tautological relations. 
From this perspective the statement that 
Barannikov-Kontsevich construction provides a solution of the 
WDVV equation looks as the simplest particular case of our theorem. 
Also it generalizes the particular cases of other low-genera tautological
relations proven in our earlier works; we replace the old technical proofs 
by a novel conceptual proof. 
\end{abstract}

\maketitle

\setcounter{tocdepth}{1}
\tableofcontents

%%%%%%%
%%%%%%%
%%%%%%%

%\newpage

\section{Introduction}

In this paper we present an attempt to formalize what
may be called a \emph{string field theory (SFT)} for 
(closed) topological strings with Hodge property.

From the very first days of string theory it was considered
as a kind of generalization of the perturbative expansion of the
quantum field theory in the (functional) integral representation.
The space of graphs with $g$ loops with metrics on edges (Schwinger proper times)
was generalized to moduli space of Riemann surfaces. Indeed,
the latter space really looks like a principle $U(1)^{n}$ bundle
over the former space near the
points of maximal degeneracy (i.e., where maximal number of handles are
pinched).

A natural question is whether there are special
string theories that degenerate exactly to quantum field
theories (may be, of the special kind). Would it happen such
theories should enjoy both finiteness of string
theory and (functional) integral description of quantum field theory.

One of the first attempts to construct a theory of this type was done by 
Zwiebach in~\cite{zwi}. He divided
the moduli space into two regions: the internal piece and the boundary. 
He observed that surfaces representing
the boundary region may be constructed from those representing the internal
piece by gluing them with the help of cylinders (with flat metric).
Therefore, he proposed to take integrals over the moduli spaces
in two steps: first, to take an integral over the internal pieces,
such that this would produce vertices, and then take an integral
along metrics on cylinders, that would exactly reproduce integral
along the Schwinger parameters on graphs in QFT prescription.

In this approach, he came with the infinite number of vertices of different
internal genera and with different number of external legs. However,
he observed that such vertices satisfy quadratic relations
that where a quantum version of some infinity-structure.
At that time community of theoretical physicists seemed not to 
be impressed by the Lagrangian
with infinite number of (almost uncomputable\footnote{
Note, that computation of an integral over a subspace
with a boundary is harder than that one over a compact space.})
vertices. 

The next attempt was done by Witten~\cite{wit-na}. He assumed that in topological
string theories there may be a limit in the space
of two-dimensional theories such that the measure of
integration goes to the vicinity of the points of maximal degeneration.
In the type B theories such limit seems to be the large volume limit
of the target space; 
this motivated Witten's Chern-Simons-like
representation for the topological string theory. 
This approach was
further developed by Bershadsky, Cecotti, Ooguri, and Vafa in~\cite{bcov}.
We note that the tropical limit of Gromov-Witten
 theory~\cite{mikha}
(type A topological strings) seems to realize the same QFT degeneration
of string theory. Indeed, the tropical limit of a Riemann surface mapped to
a toric variety is represented by the graph mapped to
the moment map domain.

In the development of topological string theory it became clear that the
proper object is not just a measure on the moduli space of complex
structures of Riemann surfaces, but rather a differential form on this
space. In original formulation these differential
forms were assigned to the tensor algebra of cohomology of some complex;
such objects are called Gromov-Witten invariants.
We say that Gromov-Witten invariants are QFT-like if
the differential forms of non-zero degree have support only in a
vicinity of the points of maximal degeneration.

We generalized the definition of Gromov-Witten invariants in~\cite{ls} by lifting it
from the cohomology of a complex to the full complex. Such generalization
involved enlargement of the moduli space from
Deligne-Mumford space to Kimura-Stasheff-Voronov space~\cite{ksv},
and we called it Zwiebach invariants (in fact, some pieces of this 
construction appeared earlier in~\cite{zwi} and~\cite{get}).
The complex of states 
involved in the definition of
Zwiebach invariants is a bicomplex due to the action of the second
differential. The second differential represents the 
substitution of a special vector
field corresponding to the constant rotation of
the phase of the local coordinate at a marked point
into differential forms on the Kimura-Stasheff-Voronov space.

Once we have some Zwiebach invariants, it is possible to produce new Zwiebach 
invariants by contraction of an acyclic Hodge sub-bicomplex. In fact, it is one
of the main properties of Zwiebach invariants. 
Consider a sub-bicomplex, where
these two differentials act freely. We call it Hodge contractible
bicomplex.  The operation of
contraction of a Hodge contractible bicomplex turns Zwiebach invariants
into induced Zwiebach invariants on the coset with respect to contactible
bicomplex. Induced Zwiebach invariants are differential forms whose support
is a union of the support of the initial Zwiebach invariants and small
neighbourhoods of the points of maximal degeneration. This procedure
is a generalization from intervals to cylinders of the procedure of induction of
$L_\infty$-structures, see e.~g.~\cite{schaetz, mnev}.

This way we can obtain QFT-like Gromov-Witten invariants.
We just should start with Zwiebach invariants that have (in some suitable sense) 
no support inside the Kimura-Stasheff-Voronov spaces. In fact, it is even enough
to consider a weaker condition, motivated by applications. That is, usually
people consider the integrals of Gromov-Witten invarians only over the
tautological classes in the moduli space of curves.
So, we call a set of Zwiebach invariants 
vertex-like if the integral over the Kimura-Stasheff-Voronov
spaces of any their non-zero component multipled by the pullback
(from the Deligne-Mumford space) of
any tautological class vanishes.

Consider vertex-like Zwiebach invariants. Assume we contract a
Hodge contractible bicomplex down to cohomology. We obtain
differential forms on the Deligne-Mumford space, such that the integral
of the product of any such form of non-zero degree with any tautological 
class vanishes the interior of the moduli space.
Integrals of such forms over the moduli spaces turn out to be sums over graphs
(corresponding to degenerate Riemann surfaces). They resemble Feynman diagramms,
and generation function for the integral over moduli spaces resemble diagrammatic
expansion of perturbative quantum field theory. 

In this paper, we don't construct examples
of vertex-like Zwiebach invariants (we are going to do this explicitly in a
future publication, as well as the corresponding theory for the spaces introduced
in~\cite{lm1,lm2}). Rather we conjecture that they exist and study the
consequences of this assumption. We call the emerging construction
the \emph{Hodge field theory}, and now we will explain it in some detail. 

First of all, degree zero parts of vertex-like Zwiebach invariants
induce the structure of homotopy cyclic Hodge algebra 
on the target complex~\cite{ls}.
We remind that a cyclic Hodge algebra is just a Hodge dGBV-algebra with one additional
axiom ($1/12$-axiom, see below). 

In fact, this structure is interesting by
itself, without any reference to Zwiebach invariants.
It has first appeared  in the paper of Barannikov and Kontsevich~\cite{bk}; 
it captures the properties of
polyvector fields on Calabi-Yau manifolds.  More examples of dGBV algebras
are studied in~\cite{mer} and~\cite{man3c}.
It is possible to understand the structure of dGBV-algebra as a natural
generalization of the algebraic structure studied in~\cite{los}.

In the Hodge field theory construction we consider only a particular case, 
where we obtain axioms of
of a cyclic Hodge algebra itself, not up to homotopy.
We are aware of the fact that demanding existance of vertex-like
Zwiebach invariants simultaneously with vanishing homotopy
piece of cyclic Hodge algebra conditions may be too restrictive,
and while considering only those relations that lead to
axioms of cyclic Hodge algebra may be too weak, however we proceed.

In the Hodge field theory construction we define 
graph expressions for the analogues of 
Gromov-Witten invariants multiplied by tautological classes
using \emph{only cyclic Hodge algebra data}.
We call them Hodge field theory correlators. The corresponding action 
of the Hodge field theory is written down explicitely in Section~\ref{action}.

Our main result is the proof that  the Hodge field theory correlators 
satisfy all universal equations that follow from relations among tautological
classes in cohomology.

The first result of this kind is due to Barannikov and Kontsevich. They have
noticed that there is a solution of the WDVV equation that is associated to a
dGBV-algebra (this solution is the critical value of the BCOV action~\cite{bcov}, 
see~\cite[Appendix]{bk} and~\cite[Appendix]{ls}). Later, we reproved this
in~\cite{ls}. Then, in~\cite{ls,s,ss,shn} we proved some other low-genera
universal equations. Here we generalize all these result and put all
calculations done before in a proper framework.

In particular, the main problem for us was to define a graph expression in tensors of a
cyclic Hodge algebra that corresponds to the full Gromov-Witten potential with descendants.
The first steps were done in~\cite{s,ss}, where we introduced the definition of
descendants at one point in Hodge field theory (mostly for combinatorial reasons). 
But then we observed that it is a part of a natural definition of potential with
descendants in cyclic Hodge algebras that appears as a special case of degeneration of 
vertex-like Zwiebach invariants multiplied by tautological classes. 

In this paper, we present and study this construction. We prove in a completely
algebraic way that Hodge field theory correlators satisfy
the same equations as a Gromov-Witten potential: string, dilaton, and the whole
system of PDEs coming from tautological relations in the cohomology of the moduli
space of curves (see also~\cite{shn} for some preliminary results). 
In what follows we will not only present the proof but also will do our best
relating algebraic definitions and statements on Hodge field theory
to analoguous constructions and theorems in the theory of Zwiebach invariants.

\subsection{Organization of the paper}

In Section~\ref{sec-GW} we remind all necessary facts about the axiomatic Gromov-Witten theory. 
In Section~\ref{sec-Z} we define Zwiebach invariants and explain the motivation to consider 
the sums over graphs in cyclic Hodge algebras. In Section~\ref{sec-defpot} we define cyclic Hodge algebras and 
the corresponding descendant potential. In Section~\ref{sec-proper} we state the main properties 
of the descendant potential in cyclic Hodge algebras, and the rest of the paper is devoted to the proofs.

\subsection{Acknowledgments}

A.L. was supported by the Russian Federal Agency of Atomic Energy and by the grants INTAS-03-51-6346, NSh-8065.2006.2, 
NWO-RFBR-047.011.2004.026 (RFBR-05-02-89000-NWO-a), and RFBR-07-02-01161-a.

S.S. was supported by the grant SNSF-200021-115907/1.
S.S. is grateful to the participants of the Moduli Spaces program at the Mittag-Leffler Institute
(Djursholm, Sweden) for the fruitful discussions of the preliminary versions of the results of this paper. The remarks of C.~Faber, O.~Tommasi, and D.~Zvonkine were especially helpful.

I.S. was supported by the grant RFBR-06-01-00037.

%%%%%%%%%%%%%%%%%%%%%%%%%%%%%%%%%%%%%%%%%%%%%%%%%%%%%%%%%%%%%%%%%%%%%%%%%%%%%
%%%%%%%%%%%%%%%%%%%%%%%%%%%%%%%%%%%%%%%%%%%%%%%%%%%%%%%%%%%%%%%%%%%%%%%%%%%%%
%%%%%%%%%%%%%%%%%%%%%%%%%%%%%%%%%%%%%%%%%%%%%%%%%%%%%%%%%%%%%%%%%%%%%%%%%%%%%
%%%%%%%%%%%%%%%%%%%%%%%%%%%%%%%%%%%%%%%%%%%%%%%%%%%%%%%%%%%%%%%%%%%%%%%%%%%%%
%%%%%%%%%%%%%%%%%%%%%%%%%%%%%%%%%%%%%%%%%%%%%%%%%%%%%%%%%%%%%%%%%%%%%%%%%%%%%

\section{Gromov-Witten theory} \label{sec-GW}

In this section we remind what is Gromov-Witten theory and explain its basic
properties that we are going to reproduce in Hodge field theory construction.

\subsection{Gromov-Witten invariants}

Let us fix a finite dimensional vector space $H_0$ over $\C$ together with the
choice of a homogeneous basis $H_0=\langle e_1,\dots,e_s \rangle$ and a non-degenerate scalar
product $\eta_{ij}=(\cdot,\cdot)$ on it. Let $e_1$ be a distinguished even element of the basis.

Consider the moduli spaces of curves $\oM_{g,n}$. On each $\oM_{g,n}$ we take
a differential form $\Omega_{g,n}$ of mixed degree with values in
$H_0^{\otimes n}$. The whole system of forms $\{\Omega_{g,n}\}$ is called 
\emph{Gromov-Witten invariants}, if it satisfies the axioms~\cite{km,man}:

\begin{enumerate}

\item There are two actions of the symmetric group $S_n$ on $\Omega_{g,n}$.
First, we can relabel the marked points on curves in $\oM_{g,n}$; second, we can
interchange the factors in the tensor product $H_0^{\otimes n}$. We require that
$\Omega_{g,n}$ is equivariant with respect to these two actions of $S_n$. In
other words, one can think that each copy of $H_0$ in the tensor product is
assigned to a specific marked point on curves in $\oM_{g,n}$.

\item The forms must be closed, $d\Omega_{g,n}=0$.

\item\label{pull-back} Consider the mapping $\pi\colon\oM_{g,n+1}\to\oM_{g,n}$ forgetting the last
marked point. Then the correspondence between $\Omega_{g,n}$ and
$\Omega_{g,n+1}$ is given by the formula
\begin{equation}
\pi^*\Omega_{g,n}=\left(\Omega_{g,n+1},e_1\right).
\end{equation}
The meaning of the right hand side is the following. We want to turn a
$H_0^{\otimes n+1}$-valued form into a $H_0^{\otimes n}$-valued one. So, we take the copy of $H_0$
corresponding to the last marked point and contract it with the vector $e_1$
using the scalar product.

\item\label{factor} 
Consider an irreducible boundary divisor in $\oM_{g,n}$, whose generic point is represented by a
two-component curve. It is the image of a natural mapping
$\sigma\colon\oM_{g_1,n_1+1}\times\oM_{g_1,n_2+1}\to\oM_{g,n}$, where
$g=g_1+g_1$ and $n=n_1+n_2$. We require that
\begin{equation}
\sigma^*\Omega_{g,n}=\left(\Omega_{g_1,n_1+1}\wedge\Omega_{g_2,n_2+1},\eta^{-1}\right).
\end{equation}
Here on the right hand side we contract with a scalar product the two copies of
$H_0$ that correspond to the node.

In the same way, consider the divisor of genus $g-1$ curves with
one self-intersection. It is the image of a natural mapping 
$\sigma\colon\oM_{g-1,n+2}\to\oM_{g,n}$. In this case, we require that 
\begin{equation}
\sigma^*\Omega_{g,n}=\left(\Omega_{g-1,n+2},\eta^{-1}\right).
\end{equation}
As before, we contract two copies of $H_0$ corresponding to the node.

\item
We also assume that $\left(\Omega_{0,3},e_1\otimes e_i\otimes e_j\right)=\left(e_i,e_j\right)=\eta_{ij}$.

\end{enumerate}

\subsection{Gromov-Witten potential}

Let us associate to each $e_i$ the set of formal variables $T_{n,i}$,
$n=0,1,2\dots$. By $F_g$ denote the formal power series in
these variables defined as 
\begin{equation}
F_g:=
\sum_{n}\frac{1}{n!}\sum_{a_1,\dots,a_n\geq 0}\int_{\oM_{g,n}}
\left(\Omega_{g,n}\prod_{i=1}^n\psi_i^{a_i},
\bigotimes_{i=1}^n\sum_{j=1}^s e_jT_{a_i,j}\right).
\end{equation}
The first sum is taken over $n\geq 3$ for $g=0$, $n\geq 1$ for $g=1$, 
and $n\geq 0$ for $g\geq 2$. On the right hand side, we contract each 
copy of $H_0$ with the factor of the tensor product associated to the 
same marked point.

The formal power series $\F:=\exp(\sum_{g\geq 0} \hbar^{g-1}F_g)$ is called 
Gromov-Witten potential associated to the system of Gromov-Witten invariants
$\{\Omega_{g,n}\}$. The coefficients of $F_g$, $g\geq 0$, are called correlators 
and denoted by
\begin{equation}
\langle \tau_{a_1,i_1}\dots \tau_{a_n,i_n} \rangle_g:=
\int_{\oM_{g,n}}
\left(\Omega_{g,n}\prod_{j=1}^n\psi_j^{a_j},
\bigotimes_{j=1}^n e_{i_j} \right).
\end{equation}
Vectors $e_{i_1},\dots,e_{i_n}$ are called primary fields.

The main properties of GW potentials come from geometry of the moduli space 
of curves. First, one can prove that coefficients of $\F$ satisfy string and dilaton equations:
\begin{align}
\langle \tau_{0,1}\prod_{j=1}^n \tau_{a_j,i_j} \rangle_g
&=\sum_{j=1}^n \langle \tau_{a_j-1,i_j}\prod_{k\not=j} \tau_{a_k,i_k} \rangle_g;
\label{string}\\
\langle \tau_{1,1}\prod_{j=1}^n \tau_{a_j,i_j} \rangle_g
&=(2g-2+n) \langle \prod_{j=1}^n \tau_{a_j,i_j} \rangle_g.
\end{align}

The string equation is a corollary of the fact that $\pi^*\psi_j=\psi_j-D_j$; 
here $\pi\colon\oM_{g,n+1}\to\oM_{g,n}$ is the projection forgetting the last marked point
and $D_j$ is the divisor in $\oM_{g,n+1}$ whose generic point is represented by a two-component 
curve with one node such that one component has genus $0$ and contains exactly two marked points, 
the $i$-th and the $(n+1)$-th ones. It is assumed that $\sum_{j=1}^n a_j > 0$.

The dilaton equation is a corollary of the fact that, in the same notations, $\pi_*\psi_{n+1}=2g-2+n$.
Of course, we assume that $2g-2+n > 0$.

Second, any relation in the cohomology of $\oM_{g,n}$ among natural $\psi$-$\kappa$-strata gives 
a relation for the correlators. Let us explain this in more detail.

\subsection{Tautological relations}

\subsubsection{Stable dual graphs}
The moduli space of curves $\oM_{g,n}$~\cite{hamo} has a natural stratification 
by the topological type of stable curves. We can combine natural strata with 
$\psi$-classes at marked points and at nodes and $\kappa$-classes on the moduli 
spaces of  irreducible components. These objects are called $\psi$-$\kappa$-strata.

A convenient way to describe a $\psi$-$\kappa$-stratum in $\oM_{g,n}$ is the 
language of stable dual graphs. Take a generic curve in the stratum. To each 
irreducible component we associate a vertex marked by its genus. To each node 
we associate an edge connecting the corresponding vertices (or a loop, if it is 
a double point of an irreducible curve). If there is a marked point on a component, 
then we add a leaf at the corresponding vertex, and we label leaves in the same 
way as marked points. If we multiply a stratum by some $\psi$-classes, then we just 
mark the corresponding leaves or half-edges (in the case when we add $\psi$-classes at nodes) 
by the corresponding powers of $\psi$. Also we mark each vertex by the $\kappa$-class 
associated to it.

Let us remark that by $\kappa$-classes on $\oM_{g,n}$ we mean the classes 
\begin{equation}
\kappa_{k_1,\dots,k_l}:=\pi_*\left(\prod_{j=1}^{l}\psi_{n+j}^{k_j+1}\right),
\end{equation}
where $\pi\colon\oM_{g,n+l}\to\oM_{g,n}$ is the projections forgetting the last $l$ 
marked points. It is just another additive basis in the ring generated by the ordinary 
$\kappa$-classes ($\kappa_k$, $k\geq 1$, in our notations). The basic properties of these classes 
are stated in~\cite{fsz}.

\subsubsection{Integrals over $\psi$-$\kappa$-strata}\label{integrals-strata}

Using the properties of GW invariants, one can express the integral of $\Omega_{g,n}$ over a 
$\psi$-$\kappa$-stratum $S$ in terms of correlators. 

Consider a special case, when $S$ is represented by a two-vertex graph with no $\psi$- and 
$\kappa$-classes. Then, according to axiom~\ref{factor}, the integral of $\Omega_{g,n}$
is the product of integrals of $\Omega_{g_1,n_1+1}$ and $\Omega_{g_2,n_2+1}$ over the moduli 
spaces corresponding to the vertices, contracted by the scalar product:
\begin{align}
\int_{S}\left(\Omega_{g,n},\bigotimes_{j=1}^n e_{i_j}\right)
&=
\int_{\oM_{g_1,n_1+1}}
\left(\Omega_{g_1,n_1+1},
\bigotimes_{j\in J_1}e_{i_j}\otimes e_{i'}\right) \\
&\eta^{i'i''}
\int_{\oM_{g_2,n_2+1}}
\left(\Omega_{g_2,n_2+1},
\bigotimes_{j\in J_2}e_{i_j}\otimes e_{i''}\right).\notag
\end{align}
Here we assume that the genus of one component of a generic curve in $S$ is
$g_1$ and $n_1$ marked points with labels $j\in J_1$, $|J_1|=n_1$, are on this
component. The other component has genus $g_2$ and $n_2$ marked points with
labels $j\in J_2$, $|J_2|=n_2$, lie on it. Of course, $g_1+g_2=g$, $n_1+n_2=n$.

Now consider a special case, when $S$ is represented by a one-vertex graph with $\psi$- and 
$\kappa$-classes. Let us assign a vector in the basis of $H_0$ to each leaf (to each marked 
point). Then, according to axiom~\ref{pull-back} the integral 
\begin{equation}
\int_{\oM_{g,n}}\left(\Omega_{g,n}\prod_{j=1}^n\psi_j^{a_j} \kappa_{b_1,\dots,b_k},\bigotimes_{j=1}^n e_{i_j}\right)
\end{equation}
is equal to
\begin{equation}
\int_{\oM_{g,n+k}}
\left(\Omega_{g,n+k}\prod_{j=1}^n\psi_j^{a_j}
\prod_{j=1}^k\psi_{n+j}^{b_j+1},
\bigotimes_{j=1}^n e_{i_j}\otimes e_1^{\otimes k}\right).
\end{equation}

Combining these two special cases one can obtain an expression in correlators that 
corresponds to an arbitrary $\psi$-$\kappa$-stratum.

\subsubsection{Relations for correlators}\label{rel-for-corr}

Suppose that we have a linear combination $L$ of $\psi$-$\kappa$-strata that is equal to $0$ in the 
cohomology of $\oM_{g,n}$ (a tautological relation). Since $d\Omega_{g,n}=0$, the integral of 
$\left(\Omega_{g,n},\bigotimes_{j=1}^ne_{i_j}\right)$ over $L$ is equal to zero,
for an arbitrary choice of primary fields. This gives an equation for correlators.

Usually, one consider also the pull-backs of $L$ to $\oM_{g,n+n'}$, $n'\geq 0$,
multiplied by arbitrary monomials of $\psi$-classes. Of course, they are also
represented as vanishing linear combinations of $\psi$-$\kappa$-strata. 
This gives a system of PDEs
for the formal power series $F_g$, $g\geq 0$. For the detailed description of the
correspondence between tautological relations and universal PDEs for GW
potentials see, e.~g., \cite{g} or~\cite{fsz}.

There are $8$ basic tautological relations known at the moment: WDVV, Getzler, 
Belorousski-Pandharipande,  and topological recursion relations in $\oM_{0,4}$, 
$\oM_{1,1}$, $\oM_{2,1}$, $\oM_{2,2}$, $\oM_{3,1}$~\cite{g,g2,bp,kl}.

%%%%%%%%%%%%%%%%%%%%%%%%%%%%%%%%%%%%%%%%%%%%%%%%%%%%%%%%%%%%%%%%%%%%%%%%%%%%%
%%%%%%%%%%%%%%%%%%%%%%%%%%%%%%%%%%%%%%%%%%%%%%%%%%%%%%%%%%%%%%%%%%%%%%%%%%%%%
%%%%%%%%%%%%%%%%%%%%%%%%%%%%%%%%%%%%%%%%%%%%%%%%%%%%%%%%%%%%%%%%%%%%%%%%%%%%%
%%%%%%%%%%%%%%%%%%%%%%%%%%%%%%%%%%%%%%%%%%%%%%%%%%%%%%%%%%%%%%%%%%%%%%%%%%%%%
%%%%%%%%%%%%%%%%%%%%%%%%%%%%%%%%%%%%%%%%%%%%%%%%%%%%%%%%%%%%%%%%%%%%%%%%%%%%%

\section{Zwiebach invariants}\label{sec-Z}

In Gromov-Witten theory (and also in topological string theory) the
Gromov-Witten invariants is usually a structure on the cohomology
of a target manifold (the space $H_0$) of on the cohomology of a complex of some
other gometric origin. We have introduced
the notion of Zwiebach invariants in~\cite{ls} in order to formalize in a
convenient way what physicists mean by topological conformal quantum 
field theory at the level of a complex rather than at the level of the
cohomology. 

The very general principles of homological algebra imply that
algebraic stuctures
on the cohomology are often induced by some fundamental structures on a full
complex (the standard example is the induction of the infinity-structures
from differential graded algebraic structures).
Such induction usually can be represented as a sum over trees with 
vertices corresponding to fundamental operations and edges corresponding 
to the homotopy that contracts the complex to its cohomology.

We would like to stress that Gromov-Witten invariants also can be considered as
an induced structure on the cohomology of a complex. In this case, the
fundamental structure on the whole complex is determined by Zwiebach invariants.

We are able to associate some structure on a bicomplex with a special
compactification of the moduli space curves (Kimura-Sta\-sheff-Voronov
compactification). So, complexes are replaced by bicomplexes, where 
the second differential reflects the rotation of attached cylinders (or
circles). This is an appearance of the string nature of the problem.

As an induced structure we indeed obtain a Gromov-Witten-type theory that, under
some additional assumptions, can be presented in terms of a sum over graphs.
Below we explain the whole construction, following~\cite{ls} and with some
additional details. 

\subsection{Kimura-Stasheff-Voronov spaces}

We remind the construction of the Kimura-Sta\-sheff-Voronov compactification 
$\overline{\mathcal{K}}_{g,n}$ of the moduli space of curves of 
genus $g$ with $n$ marked point. It is a real blow-up of $\oM_{g,n}$; 
we just remember the relative angles at double points. We can also 
choose an angle of the tangent vector at each marked point; this 
way we get the principal $U(1)^n$-bundle over $\overline{\mathcal{K}}_{g,n}$. 
We denote the total space of this bundle by $\overline{\mathcal{S}}_{g,n}$.

There are also the standard mappings between different spaces $\overline{\mathcal{S}}_{g,n}$. 
First, one can consider the projection 
$\pi\colon\overline{\mathcal{S}}_{g,n+1}\to\overline{\mathcal{S}}_{g,n}$
forgetting the last marked point. Suppose that under the projection we 
have to contract a sphere that contains the points $x_i$, $x_{n+1}$, and a node.
Denote the natural coordinates on the circles corresponding to $x_i$ and a node on a curve in 
$\overline{\mathcal{S}}_{g,n+1}$ by $\phi_i$ and $\theta$. Let $\tilde\phi_i$ be a coordinate on 
the circle corresponding to $x_i$ in $\overline{\mathcal{S}}_{g,n}$. 
Then $\tilde\phi_i=\phi_i+\theta$ under the projection $\pi$. In the same way, if we contract 
a sphere that contains two nodes and $x_{n+1}$, then $\tilde\theta=\theta_1+\theta_2$, where $\theta_1$ 
and $\theta_2$ are the coordinates on the circles corresponding to the two nodes of a curve in 
$\overline{\mathcal{S}}_{g,n+1}$ and $\tilde\theta$ is a coordinate on the circle at the resulting 
node in $\overline{\mathcal{S}}_{g,n}$

In the same way, when we consider the mappings 
$\sigma\colon\overline{\mathcal{S}}_{g_1,n_1+1}\times\overline{\mathcal{S}}_{g_1,n_2+1}
\to\overline{\mathcal{S}}_{g,n}$ representing the natural boundary components of 
$\overline{\mathcal{S}}_{g,n}$, we have $\theta=\phi_{n_1+1}+\phi_{n_2+1}$, where 
$\phi_{n_1+1}$ and $\phi_{n_2+1}$ are the coordinates on the circles corresponding the points 
that are glued by $\sigma$ into the node and $\theta$ is the coordinate on the circle at the node.
For the mapping $\sigma\colon\overline{\mathcal{S}}_{g-1,n+2}
\to\overline{\mathcal{S}}_{g,n}$ we also have $\theta=\phi_{n+1}+\phi_{n+2}$ with the same notations.

\subsection{Zwiebach invariants}

Let us fix a Hodge bicomplex $H$ with two differentials denoted by $Q$ and $G_{-}$ 
and with an even scalar product $\eta=(\cdot ,\cdot)$ invariant under the differentials:
\begin{equation}
(Qv,w)=\pm(v,Qw),\quad (G_-v,w)=\pm(v,G_-w). 
\end{equation}
The Hodge property means that 
\begin{equation}
H=H_0\oplus
\bigoplus_{\alpha} 
\langle 
e_\alpha, Qe_\alpha, G_-e_\alpha, QG_-e_\alpha
\rangle,
\end{equation}
where $QH_0=G_-H_0=0$ and $H_0$ is orthogonal to $H_4$.

Below we consider the action of $Q$ and $G_-$ on $H^{\otimes n}$. 
We denote by $Q^{(k)}$ and $G_-^{(k)}$ the action of $Q$ and $G_-$ 
respectively on the $k$-th component of the tensor product.

On each $\overline{\mathcal{S}}_{g,n}$ we take
a differential form $C_{g,n}$ of the mixed degree with values in
$H_0^{\otimes n}$. The whole system of forms $\{C_{g,n}\}$ is called 
\emph{Zwiebach invariants}, if it satisfies the axioms:

\begin{enumerate}

\item $\Omega_{g,n}$ is $S_n$-equivariant.

\item $(d+Q)\Omega_{g,n}=0$, $Q=\sum_{i=1}^nQ^{(i)}$.

\item $(G_-^{(k)}+\imath_k)C_{g,n}=0$ for all $1\leq k\leq n$ (we denote by $\imath_k$ the substitution of the vector field generating the action on $\overline{\mathcal{S}}_{g,n}$ of the $k$-th copy of $U(1)$); $C_{g,n}$ is invariant under the action of $U(1)^n$;

\item $\pi^*C_{g,n}=\left(C_{g,n+1},e_1\right)$, where
$\pi\colon\overline{\mathcal{S}}_{g,n+1}\to\overline{\mathcal{S}}_{g,n}$ is the mapping forgetting the last
marked point. 

\item $\sigma^*C_{g,n}=\left(C_{g_1,n_1+1}\wedge C_{g_2,n_2+1},\eta^{-1}\right)$,
where $\sigma\colon\overline{\mathcal{S}}_{g_1,n_1+1}\times
\overline{\mathcal{S}}_{g_1,n_2+1}\to\overline{\mathcal{S}}_{g,n}$ represents the boundary component.
In the same way, 
$\sigma^*C_{g,n}=\left(C_{g-1,n+2},\eta^{-1}\right)$
for the mapping $\sigma\colon\overline{\mathcal{S}}_{g-1,n+2}\to\overline{\mathcal{S}}_{g,n}$.

\item $\left(C_{0,3},e_1\otimes v_\alpha\otimes v_\beta\right)=\left((Id+d\phi_2 G_-)v_\alpha,(Id+d\phi_3 G_-)v_\beta\right)$, $\phi_2$ and $\phi_3$ are the coordinates on the circles at the corresponding points.

\end{enumerate}

Zwiebach invariants on the bicomplex with zero differentials 
determine Gromov-Witten invariants. Indeed, in this case the factorization property 
implies that $\{C_{g,n}\}$ is lifted from 
the blowdown of Kimura-Stasheff-Voronov spaces, i.e. it is determined 
by a set of forms on Deligne-Mumford spaces.
Then it is easy to check that this system of forms satisfied all axioms of Gromov-Witten invariants.

\subsection{Induced Zwiebach invariants}
Induced Zwiebach invariants are obtained by the contraction of $H_4$.
We denote by $G_+$ the contraction operator. This means that $G_+H_0=0$, $\Pi=\{Q,G_+\}$ is the projection to $H_4$ along $H_0$, $\{G_+,G_-\}=0$, and $(G_+v,w)=\pm(v,G_+w)$.

We construct an induced Zwiebach form $C^{ind}_{g,n}$ 
on a homotopy equivalent modification $\tilde{\mathcal{S}}_{g,n}$ of the space $\overline{\mathcal{S}}_{g,n}$. 
At each boundary component $\gamma$ we glue the cylinder $\gamma\times[0,+\infty]$ such that $\gamma$ in $\overline{\mathcal{S}}_{g,n}$ is identified with $\gamma\times\{0\}$ in the cylinder.

So, we have the mappings $\tilde\sigma\colon\tilde{\mathcal{S}}_{g_1,n_1+1}\times
\tilde{\mathcal{S}}_{g_1,n_2+1}\times[0,+\infty]\to\tilde{\mathcal{S}}_{g,n}$ and 
$\tilde\sigma\colon\tilde{\mathcal{S}}_{g-1,n+2}\times[0,+\infty]\to\tilde{\mathcal{S}}_{g,n}$
representing the boundary components with glued cylinders.
We take a form $C_{g,n}$, restrict it to $H_0^{\otimes n}$, and extend it to the glued cylinder 
by the rule that
\begin{equation}
\tilde\sigma^*C^{ind}_{g,n}=\left(C^{ind}_{g_1,n_1+1}\wedge C^{ind}_{g_2,n_2+1}, [e^{-t\Pi-dt\cdot G_+}] \right)
\end{equation}
in the first case and 
\begin{equation}
\tilde\sigma^*C^{ind}_{g,n}=\left(C^{ind}_{g-1,n+2}, [e^{-t\Pi-dt\cdot G_+}] \right)
\end{equation}
in the second case, where $[e^{-t\Pi-dt\cdot G_+}]$ is the bivector obtained from the operator $e^{-t\Pi-dt\cdot G_+}$, $t$ is the coordinate on $[0,+\infty]$. This determines $C^{ind}_{g,n}$ completely.

Now it is a straightforward calculation to check that the forms $C^{ind}_{g,n}$ are $(d+Q)$-closed and satisfy the factorization property when restricted to the strata $\gamma\times\{+\infty\}$.

\subsection{Induced Gromov-Witten theory}
The induced Zwiebach invariants determine Gromov-Witten invariants. 
The correlators of the corresponding Gromov-Witten 
potential are given by the integrals over the fundamental cycles of 
$\tilde{\mathcal{K}}_{g,n}$ (we just forget the circles at marked 
points in $\tilde{\mathcal{S}}_{g,n}$) of the forms $C^{ind}_{g,n}\prod_{i=1}^n\psi_i^{a_i}$.

In fact, the fundamental class of $\tilde{\mathcal{K}}_{g,n}$ is represented as a sum over 
all irreducible boundary strata in $\oM_{g,n}$. Indeed, a boundary stratum $\gamma$ in 
$\oM_{g,n}$ has real codimension equal to the doubled number of the nodes of its generic 
curve. But then we add in $\tilde{\mathcal{K}}_{g,n}$ a real two-dimensional cylinder for 
each node. A simple explicit calculation allows to express the integral over the component 
of the fundamental cycle of $\tilde{\mathcal{K}}_{g,n}$ corresponding to $\gamma$. It splits 
into the integrals of the initial Zwiebach invariants (multiplied by $\psi$-classes) over 
the moduli spaces corresponding to the irreducible components of curves in $\gamma$; they 
are contracted with the bivectors $[G_-G_+]$ (obtained from the operator $G_-G_+$ via the 
scalar product) corresponding to the nodes according to the topology of curves in $\gamma$.

So, we represent the correlators of the induced Gromov-Witten theory as sums over graphs. 
Then one can observe that 
$C_{0,3}$ determines a multiplication on $H$. Topology of the spaces $S_{0,4}$
and $S_{1,1}$ implies that the  whole algebraic structure that we obtain on $H$
is  the structure of cyclic Hodge algebra up to $Q$-homotopy, see~\cite{ls}. Let
 us assume that the initial system of Zwiebach invariants is simple enough,
 i.~e.,  it induces the explicit structure of cyclic Hodge algebra on $H$ and
 only  the integrals of the zero-degree parts of the initial Zwiebach invariants
  (multiplied by $\psi$-classes) are non-vanishing on fundamental cycles. In
	this  case, the induced Gromov-Witten potential can be described in very
	simple  algebraic terms. It is the motivation of the definition of the
Hodge field theory construction
given in the next section.

%%%%%%%%%%%%%%%%%%%%%%%%%%%%%%%%%%%%%%%%%%%%%%%%%%%%%%%%%%
%%%%%%%%%%%%%%%%%%%%%%%%%%%%%%%%%%%%%%%%%%%%%%%%%%%%%%%%%%
%%%%%%%%%%%%%%%%%%%%%%%%%%%%%%%%%%%%%%%%%%%%%%%%%%%%%%%%%%
%%%%%%%%%%%%%%%%%%%%%%%%%%%%%%%%%%%%%%%%%%%%%%%%%%%%%%%%%%
%%%%%%%%%%%%%%%%%%%%%%%%%%%%%%%%%%%%%%%%%%%%%%%%%%%%%%%%%%

\section{Construction of correlators in Hodge field theory}\label{sec-defpot}

In this section, we describe in a very formal algebraic way the sum over graphs
obtained as an expression for the Gromov-Witten potential induced from Zwiebach
invariants in the previous section. 

\subsection{Cyclic Hodge algebras}\label{cyclic Hodge algebras}
In this section, we recall the definition of cyclic Hodge
dGBV-alge\-bras~\cite{ls,s,ss,man} 
(cyclic Hodge algebras, for short). A supercommutative associative $\C$-algebra $H$ with
unit is called cyclic Hodge algebra, 
if there are two odd linear operators $Q,G_-\colon H\to H$ and an even linear function 
$\int\colon H\to\C$ called integral. They must satisfy the following axioms:

\begin{enumerate}

\item $(H,Q,G_-)$ is a bicomplex: 
\begin{equation}
Q^2=G_-^2=QG_-+G_-Q=0;
\end{equation}

\item $H=H_0\oplus H_4$, where $QH_0=G_-H_0=0$ and $H_4$ is represented as 
a direct sum of subspaces of dimension $4$ generated by 
$e_\alpha, Qe_\alpha, G_-e_\alpha, QG_-e_\alpha$ for some vectors 
$e\in H_4$, i.~e. 
\begin{equation}
H=H_0\oplus
\bigoplus_{\alpha} 
\langle 
e_\alpha, Qe_\alpha, G_-e_\alpha, QG_-e_\alpha
\rangle
\end{equation}
(Hodge decomposition);

\item $Q$ is an operator of the first order, it satisfies the Leibniz rule: 
\begin{equation}
Q(ab)=Q(a)b+(-1)^{\tilde a}aQ(b)
\end{equation}
(here and below we denote by $\tilde a$  the parity of $a\in H$);

\item $G_-$ is an operator of the second order, 
it satisfies the $7$-term relation:
\begin{align}
G_-(abc)& = G_-(ab)c+(-1)^{\tilde b(\tilde a+1)}bG_-(ac)
+(-1)^{\tilde a}aG_-(bc)\\
& -G_-(a)bc-(-1)^{\tilde a}aG_-(b)c
-(-1)^{\tilde a+\tilde b}abG_-(c). \notag
\end{align}

\item $G_-$ satisfies the property called $1/12$-axiom: 
\begin{equation}
str(G_-\circ a\cdot)=(1/12)str(G_-(a)\cdot)
\end{equation}
(here $a\cdot$ and $G_-(a)\cdot$ are the 
operators of multiplication by $a$ and $G_-(a)$ respectively, $str$ means
supertrace).

\end{enumerate}

Define an operator $G_+\colon H\to H$ related to the particular choice of Hodge
decomposition. We put $G_+H_0=0$, and on each subspace 
$\langle e_\alpha, Qe_\alpha, G_-e_\alpha, QG_-e_\alpha \rangle$ 
we define $G_+$ as 
\begin{align}
& G_+e_\alpha=G_+G_-e_\alpha=0, \\
& G_+Qe_\alpha=e_\alpha, \notag \\
& G_+QG_-e_\alpha =G_-e_\alpha. \notag
\end{align}
We see that $[G_-,G_+]=0$; $\Pi_4=[Q,G_+]$ is the projection to $H_4$ along $H_0$; 
$\Pi_0=\mathrm{Id}-\Pi_4$ is the projection to $H_0$ along $H_4$. 

Consider the integral $\int\colon H\to\C$. We require that
\begin{align}
& \int Q(a)b = (-1)^{\tilde a+1}\int aQ(b), \\
& \int G_-(a)b = (-1)^{\tilde a}\int aG_-(b), \notag \\
& \int G_+(a)b  = (-1)^{\tilde a}\int aG_+(b). \notag
\end{align}
These properties imply that
$\int G_-G_+(a)b=\int aG_-G_+(b)$, $\int \Pi_4(a)b=\int a\Pi_4(b)$, and
$\int \Pi_0(a)b=\int a\Pi_0(b)$.

We can define a scalar product on $H$ as $(a,b)=\int ab.$ We suppose that this scalar product 
is non-degenerate. Using the scalar product we may turn any operator $A: H \to H$
into the bivector that we denote by $[A]$.

\subsection{Tensor expressions in terms of graphs}\label{tensors}
Here we explain a way to encode some tensor expressions over an arbitrary vector space in terms of graphs.

Consider an arbitrary graph (we allow graphs to have leaves and we require vertices to be at least of 
degree $3$, the definition of graph that we use can be found in~\cite{man}). We associate a symmetric 
$n$-form to each internal vertex of degree $n$, a symmetric bivector 
to each egde, and a vector to each leaf. Then we can substitute the tensor product of all vectors in leaves 
and bivectors in edges into the product of $n$-forms in vertices, distributing the components of tensors in 
the same way as the corresponding edges and leaves are attached to vertices in the graph. This way we get a 
number.

Let us study an example:
\begin{equation}
\inspic{arb.1}
\end{equation}
We assign a $5$-form $x$ to the left vertex of this graph and a $3$-form $y$ to the right vertex. 
Then the number that we get from this graph is $x(a,b,c,v,w)\cdot y(v,w,d)$.

Note that vectors, bivectors and $n$-forms used in this construction can depend on some variables. 
Then what we get is not a number, but a function.

\subsection{Usage of graphs in cyclic Hodge algebras}\label{usage}

Consider a cyclic Hodge algebra $H$. There are some standard tensors over $H$, 
which we associate to elements of graphs below. Here we introduce the notations for these tensors.

We always assign the form
\begin{equation}\label{internal-vertex}
(a_1,\dots,a_n)\mapsto \int a_1\cdot\dots\cdot a_n
\end{equation}
to a vertex of degree $n$.

There is a collection of bivectors that will be assigned below to edges: 
$[G_-G_+]$, $[\Pi_0]$, $[Id]$, $[QG_+]$, $[G_+Q]$, $[G_+]$, and $[G_-]$. 
In pictures, edges with these bivectors will be denoted by
\begin{equation}\label{edges}
\inspic{arb.2},\quad \inspic{arb.3},\quad \inspic{arb.4},\quad \inspic{arb.5},
\quad \inspic{arb.6},\quad \inspic{arb.7},\quad \inspic{arb.8},
\end{equation}
respectively. Note that an empty edge corresponding to the bivector $[Id]$ can 
usually be contracted (if it is not a loop).

The vectors that we will put at leaves depend on some variables. Let $\{e_1,\dots,e_s\}$ be a
homogeneous basis of $H_0$. In particular, we assume that $e_1$ is the unit of
$H$. To each vector $e_i$ we associate formal variables $T_{n,i}$, 
$n\geq 0$, of the same parity as $e_i$. Then we will put at a leaf one of the vectors 
$E_n=\sum_{i=1}^s e_i T_{n,i}$, $n\geq 0$, and we will mark such leaf by the
number $n$. In our picture, an empty leaf is the same as the leaf marked by $0$.

\subsubsection{Remark}

There is a subtlety related to the fact that $H$ is a $\Z_2$-graded space. 
In order to give an honest definition we must do the following. Suppose we 
consider a graph of genus $g$. We can choose $g$ edges in such a way that 
the graph being cut at these edge turns into a tree. To each of these edges 
we have already assigned a bivector $[A]$ for some operator $A\colon H\to H$. 
Now we have to put the bivector $[JA]$ instead of the bivector $[A]$, where 
$J$ is an operator defined by the formula $J\colon a\mapsto (-1)^{\tilde a} a$.

In particular, consider the following graph (this is also an example to the 
notations given above):
\begin{equation}
\inspic{arb.9}
\end{equation}
An empty loop corresponds to the bivector $[Id]$. An empty leaf corresponds 
to the vector $E_0$. A trivalent vertex corresponds to the $3$-form given by 
the formula $(a,b,c)\mapsto\int abc$.

If we ignore this remark, then what we get is just the trace of the operator 
$a\mapsto E_0\cdot a$. But using this remark we get the supertrace of this operator.

In fact, this subtlety will play no role in this paper. It affects only some 
signs in calculations and all these signs will be hidden in lemmas shared 
from~\cite{ls,s}. So, one can just ignore this remark.

\subsection{Correlators}\label{subsec-correlators}

We are going to define the potential using correlators. Let
\begin{equation}
\langle \tau_{k_1}(V_1)\dots \tau_{k_n}(V_n) \rangle_g
\end{equation}
be the sum over graphs of genus $g$ with $n$ leaves marked by $\tau_{k_i}(V_i)$, $i=1,
\dots, n$, where $V_1,\dots,V_n$ are vectors in $H$, and $\tau_{k_i}$ are just formal
symbols.
The index of each internal vertex of these graphs is $\geq 3$; we associate to it
the symmetric form~\eqref{internal-vertex}.
There are two possible types of edges: edges marked by $[G_-G_+]$ (thick black
dots in pictures, ``heavy edges'' in the text) and edges marked by $[Id]$ (empty edges). Since 
an empty edge connecting two different vertices can be contracted, we assume that all 
empty edges are loops.

Consider a vertex of such graph. Let us describe all possible half-edges adjusted to
this vertex. There are $2g$, $g\geq 0$, half-edges coming from $g$ empty loops; $m$ 
half-edges coming from heavy edges of graph, and $l$ leaves 
marked $\tau_{k_{a_1}}(V_{a_1}), \dots, \tau_{k_{a_l}}(V_{a_l})$. Then we say that 
the type of this vertex is $(g,m;k_{a_1},\dots,k_{a_l})$. We denote the type of a 
vertex $v$ by $(g(v),m(v);k_{a_1(v)},\dots,k_{a_{l(v)}(v)})$.

Consider a graph $\Gamma$ in the sum determining the correlator
\begin{equation}
\langle \tau_{k_1}(V_1)\dots \tau_{k_n}(V_n) \rangle_g
\end{equation}
We associate to $\Gamma$ a number: we contract according to the graph structure all
tensors corresponding to its vertices, edges, and leaves (for leaves, we take
vectors $V_1,\dots,V_n$). Let us denote this number by $T(\Gamma)$.

Also we weight each graph by a coefficient which is the product of two
combinatorial constants. The first factor is equal to
\begin{equation}
V(\Gamma)=\frac{\prod_{v\in Vert(\Gamma)}2^{g(v)}g(v)!}{|\mathrm{aut}(\Gamma)|}.
\end{equation}
Here $|\mathrm{aut}(\Gamma)|$ is the order of the automorphism group of the labeled
graph $\Gamma$, $Vert(\Gamma)$ is the set of internal vertices of $\Gamma$. In other
words,
we can label each vertex $v$ by $g(v)$, delete all empty loops, and then we get a
graph with the order of the automorphism group equal to $1/V(\Gamma)$.

The second factor is equal to
\begin{equation}
P(\Gamma)=\prod_{v\in Vert(\Gamma)}\int_{\oM_{g(v),m(v)+l(v)}}\psi_1^{a_1(v)}\dots
\psi_{l(v)}^{a_{l(v)}(v)}.
\end{equation}

The integrals used in this formula can be calculated with the help of the
Witten-Kontsevich theorem~\cite{Witten,Kontsevich,Okounkov-Pandharipande,
Mirzakhani,Kazarian-Lando,Kim-Liu,Chen-Li-Liu}.

So, the whole contribution of the graph $\Gamma$ to the correlator is equal to
$V(\Gamma)P(\Gamma)T(\Gamma)$. One can check that the non-trivial contribution to 
the correlator $\langle\tau_{k_1}(V_1)\dots \tau_{k_n}(V_n) \rangle_g$ is given only 
by graphs that have exactly $3g-3+n-\sum_{i=1}^n k_i$ heavy edges.

The geometric meaning here is very clear. The number $T(\Gamma)$ comes from the
integral of the induced Gromov-Witten invariants of degree zero, while the
coefficient $V(\Gamma)P(\Gamma)$ is exactly the combinatorial interpretation of
the intersection number of $\psi_1^{k_1}\dots\psi_n^{k_n}$ with the stratum
whose dual graph is obtained from $\Gamma$ by the procedure described after the
definition of $V(\Gamma)$.

\subsection{Potential}

We fix a cyclic Hodge algebra and consider the formal power series $\F=\F(T_{n,i})$ defined as
\begin{multline}\label{potential}
\F=
\exp\left(\sum_{g=0}^\infty \hbar^{g-1} F_g\right) \\
=\exp\left(\sum_{g=0}^\infty \hbar^{g-1} \sum_{n} \frac{1}{n!} \sum_{a_1,\dots,a_n \in \Z_{\geq 0}}
\langle \tau_{a_1}(E_{a_1})\dots \tau_{a_n}(E_{a_n})\rangle_g\right).
\end{multline}

Abusing notations, we allow to mark the leaves by $\tau_{a}(E_{a})$, $E_{a}$, or
$a$; all this variants are possible and denote the same.

\subsection{Trivial example}

For example, consider the trivial cyclic Hodge algebra: $H=H_0=\langle e_1\rangle$,
$Q=G_-=0$, $\int e_1=1$.
Then $E_a=e_1\cdot t_a$, and the correlator 
$
\langle\tau_{a_1}(E_{a_1})\dots\tau_{a_n}(E_{a_n})\rangle_g
$
consists just of one graph with one vertex, $g$ empty loops, and $n$ leaves marked by $a_1,\dots,a_n$.
The explicit value of the coefficient of this graph is, by definition,
\begin{equation}\label{GWpoint}
\langle\tau_{a_1}\dots\tau_{a_n}\rangle_g:=\int_{\oM_{g,n}}
\psi_1^{a_1}\dots\psi_{n}^{a_n}.
\end{equation}
So, in the case of trivial cyclic Hodge algebra we obtain exactly the Gromov-Witten
potential of the point (i.~e., $\left(\Omega_{g,n},e_1^{\otimes n}
\right)\equiv 1)$) that we denote below by $F^{pt}$.

\subsubsection{Remark about notations}

Abusing notation, we use the same symbol $\langle~\rangle_g$ for the 
correlators in GW theory and in Hodge field theory. We hope that it 
does not lead to a confusion. For instance, 
$
\langle\tau_{a_1}(E_{a_1})\dots\tau_{a_n}(E_{a_n})\rangle_g
$
in the trivial example above is the correlator of the trivial Hodge field theory, while 
$
\langle\tau_{a_1}\dots\tau_{a_n}\rangle_g
$
is the correlator of the trivial GW theory.

%%%%%%%%%%%%%%%%%%%%%%%%%%%%%%%%%%%%%%%%%%%%%%%%%%%%%%%%%%
%%%%%%%%%%%%%%%%%%%%%%%%%%%%%%%%%%%%%%%%%%%%%%%%%%%%%%%%%%
%%%%%%%%%%%%%%%%%%%%%%%%%%%%%%%%%%%%%%%%%%%%%%%%%%%%%%%%%%
%%%%%%%%%%%%%%%%%%%%%%%%%%%%%%%%%%%%%%%%%%%%%%%%%%%%%%%%%%
%%%%%%%%%%%%%%%%%%%%%%%%%%%%%%%%%%%%%%%%%%%%%%%%%%%%%%%%%%

\section{String, dilaton, and tautological relations}\label{sec-proper}

In this section, we prove that the potential~\eqref{potential} satisfies the
same string and dilaton equations as GW potentials.

\subsection{String equation}

\begin{theorem} If $\sum_{j=1}^n a_j >0$, we have:
\begin{equation}\label{str}
\langle \tau_0(e_1)\prod_{j=1}^n\tau_{a_j}(e_{i_j})\rangle_g
=\sum_{j=1}^n \langle \tau_{a_j-1}(e_{i_j})\prod_{k\not=j} \tau_{a_k}(e_{i_k})\rangle_g;
\end{equation}
\end{theorem}

\begin{proof}

Consider a graph $\Gamma$ contributing to the correlator on the left hand side of the
string equation. The special leaf that we are going to remove is marked by
$\tau_0(e_1)$ and is attached to a vertex $v$ of genus $g_v$ (i.~e., with $g_v$
attached light loops) with $l_v$ more attached leaves labeled by indices in $I_v$,
$|I_v|=l_v$,
and $m_v$ attached half-edges coming from heavy edges and loops.

Let us remove the leaf $\tau_0(e_1)$ and change the label of one of the leaves attached to
the same vertex from $\tau_{a_j}(e_{i_j})$ to $\tau_{a_j-1}(e_{i_j})$. This way
we obtain a graph $\Gamma_j$ contributing to the $j$-th summand of the right hand side
of~\eqref{str}. We take the sum of these graphs over $j\in I_v$. Of course, we
skip the summands where $a_j=0$.

Note that this sum is not empty (if $\Gamma$ gives a non-zero contribution to the
left hand side of~\eqref{str}). Indeed, if it is 
empty, this means that $a_j=0$ for all $j\in I_v$. Therefore, since we expect that the 
contribution to $P(\Gamma)$ of the vertex $v$ on the left hand side is nonzero, it 
follows that $g_v=0$ and $m_v+l_v=2$. So, there are three possible local pictures:
\begin{equation}
\inspic{arb.10},\quad\inspic{arb.11},\quad\mbox{and}\quad\inspic{arb.12}.
\end{equation}
The first picture can be replaced with the bivector $[G_-G_+G_-G_+]$, which is
equal to zero. Therefore, $T(\Gamma)$ is also equal to zero.  In the second case, 
we also get $0$ since $G_-G_+(e_1e_j)=G_-G_+(E_0)=0$. The third picture 
is possible only when it is the whole graph, and this is in contradiction with
the assumption that $\sum_{j=1}^n a_j >0$.

Note also that $T(\Gamma)=T(\Gamma_j)$ and $V(\Gamma)=V(\Gamma_j)$ for all $j$. 
Indeed, we have just removed the leaf with the unit of the algebra, so this
can't change anything in the contraction of tensors. Therefore, $T(\Gamma)=
T(\Gamma_j)$. Also both the leaf $\tau_0(e_1)$ and the vertex $v$ 
are the fixed points of any automorphism of $\Gamma$. The same is for 
the vertex corresponding to $v$ in $\Gamma_j$. Therefore, the automorphism groups 
are isomorphic for both graphs. Since we make no changes for empty loops, it
follows that $V(\Gamma)=V(\Gamma_j)$.

Let us prove that $P(\Gamma)=\sum_{j\in I_v}P(\Gamma_j)$. Indeed, the
vertices of $\Gamma$ and $\Gamma_j$ are in a natural one-to-one correspondence. 
Moreover, the local pictures for all of them except for $v$ and its image in
$\Gamma_j$ are the same. Therefore, the corresponding intersection numbers
contributing to $P(\Gamma)$ and $P(\Gamma_j)$ are the same. The unique
difference appeares when we take the intersection numbers corresponding to 
$v$ and its images in $\Gamma_j$, $j\in I_v$. But then we can apply the string
equation~\eqref{string} of the GW theory of the point~\eqref{GWpoint}, and we see that
\begin{equation}
\int_{\oM_{g_v,k_v+l_v+1}}\prod_{j\in I_v}\psi_{o(j)}^{a_j}
=
\sum_{j\in I_v}
\int_{\oM_{g_v,k_v+l_v}}\psi_{o(j)}^{a_j-1}\prod_{k\not= j}\psi_{o(k)}^{a_k}
\end{equation}
(here $o\colon I_v\to \{1,\dots,l\}$ is an arbitrary on-to-one mapping). This
implies that $P(\Gamma)=\sum_{j\in I_v}P(\Gamma_j)$.

So, we have $V(\Gamma)P(\Gamma)T(\Gamma)=
\sum_{j\in I_v}V(\Gamma_j)P(\Gamma_j)T(\Gamma_j)$.
In order to complete the proof of~\eqref{str}, we should just notice that when we
write down this expression for all graphs contributing to the left hand side 
of~\eqref{str}, we use each graph contributing to the right hand side of~\eqref{str} 
exactly once.

\end{proof}

\subsection{Dilaton equation}

\begin{theorem} If $2g-2+n>0$, we have:
\begin{equation}\label{dil}
\langle \tau_1(e_1)\prod_{j=1}^n\tau_{a_j}(e_{i_j})\rangle_g
=(2g-2+n)\langle \prod_{j=1}^n \tau_{a_j}(e_{i_j})\rangle_g;
\end{equation}
\end{theorem}

\begin{proof}

Consider a graph $\Gamma$ contributing to the correlator on the left hand side
of~\eqref{dil}. The special leaf that we are going to remove is marked by
$\tau_1(e_1)$ and is attached to a vertex $v$ of genus $g_v$ (i.~e., with $g_v$
attached light loops) with $l_v$ more attached leaves labeled by indices in $I_v$,
$|I_v|=l_v$, and $m_v$ attached half-edges coming from heavy edges and loops.

Let us remove the leaf $\tau_1(e_1)$. We obtain a graph $\Gamma'$ contributing 
to the right hand side of~\eqref{str}. Let us prove this. Indeed, if we remove a
leaf and don't get a proper graph, it follows that we have a trivalent
vertex. Since the contribution of this vertex to $P(\Gamma)$ should be non-zero, 
it follows that the unique possible local picture is 
\begin{equation}
\inspic{arb.13}.
\end{equation}
But this picture is the whole graph, and it is in contradiction with the
condition $2g-2+n>0$.

The same argument as in the proof of the string equation shows that 
$T(\Gamma)=T(\Gamma')$ and $V(\Gamma)=V(\Gamma')$. Also, the contribution to
$P(\Gamma)$ and $P(\Gamma')$ of all vertices except for the changed one is the 
same. The change of the intersection number corresponding to the vertex $v$ is
captured by the dilaton equation~\eqref{string} of the trivial GW
theory~\eqref{GWpoint}:
\begin{equation}
\int_{\oM_{g_v,k_v+l_v+1}}\psi_{l+1}\prod_{j\in I_v}\psi_{o(j)}^{a_j}
=
(2g_v-2+k_v+l_v)\int_{\oM_{g_v,k_v+l_v}}\prod_{j\in I_v}\psi_{o(j)}^{a_j}
\end{equation}
(again, $o\colon I_v\to \{1,\dots,l\}$ is an arbitrary on-to-one mapping). 
This implies that $P(\Gamma)=(2g_v-2+k_v+l_v)P(\Gamma')$, and, therefore,
\begin{equation}
V(\Gamma)P(\Gamma)T(\Gamma)=(2g_v-2+k_v+l_v)V(\Gamma')P(\Gamma')T(\Gamma').
\end{equation}

Let us write down the last equation for all graphs $\Gamma$ contributing to the
left hand side~\eqref{dil}. Observe that any graph $\Gamma'$ contributing to the
right hand side occurs $|Vert(\Gamma')|$ times, since the leaf $\tau_1(e_1)$
could be attached to any its vertex. Therefore, any graph $\Gamma'$ contributing to the
right hand side of~\eqref{dil} appears in these equations with the coefficient
\begin{equation}
\sum_{v\in Vert(\Gamma')}\left(2g_v-2+k_v+l_v\right)=2g-2+n.
\end{equation}
This completes the proof.

\end{proof}

\subsection{Tautological equations}

As we have explained in Section~\ref{rel-for-corr}, any linear relation $L$ among 
$\psi$-$\kappa$-strata in the cohomology of
the moduli space of curves gives rise to a family of universal relations for the correlators 
of a Gromov-Witten theory.

\begin{theorem}[Main Theorem] \label{thm-tautolog} 
The system of universal relations coming from a tautological relation in the
cohomology of the moduli space of curves holds for the correlators $\langle
\prod_{j=1}^n\tau_{a_j}(e_{i_j})\rangle_g$ of cyclic Hodge algebra.
\end{theorem}

Note that some special cases of this theorem were proved in~\cite{ls,s,ss}. Our
argument below is a natural generalization of the technique introduced in these
papers. Also we are able now to give an explanation why we have managed to perform 
all our calculations there, see Remark~\ref{last-remark}.

Let us give here a brief account of the proof of this theorem. First, the
definition of correlators of the Hodge field theory can be extended to the
intersection with an arbitrary tautological class $\alpha$ of degree $K$ in the space
$\oM_{g,n}$, not only the a monomial in
$\psi$-classes. In that case, we do the following. Again, we consider the sum
over all graphs $\Gamma$ with $3g-3+n-K$ heavy edges, and the number $T(\Gamma)$
is defined as above. Instead of the coefficient $V(\Gamma)P(\Gamma)$ we use
the intersection number of $\alpha$ with the stratum
whose dual graph is obtained from $\Gamma$ by the procedure described right 
after the definition of $V(\Gamma)$. Namely, a vertex with $g$ loops is replaced
by a vertex marked by $g$. 

This definition is very natural from the point of view of Zwiebach invariants.
However, we know from Gromov-Witten theory that this extension of the notion of
correlator is unnecessary. Indeed, all integrals with arbitrary tautological
classes can be expressed in terms of the integrals with only $\psi$-classes via
some universal formulas. 

The main question is whether these universal formulas also work in Hodge field
theory. Actually, the main result that more or less immediately proves the
theorem is the positive answer to this question. 

\subsubsection{Organization of the proof}

The rest of the paper is devoted to the proof of Main Theorem, and here we would like to overview it here. 

In Section 6, we study the structure of graphs that can appear in formulas for the correlators of Hodge field theory. We prove that if $T(\Gamma)\not=0$ and there is at least one heavy edge in $\Gamma$, then all vertices have genus $\leq 1$, i.~e., there is at most one empty loop at any vertex. This basically means that in calculations we'll have to deal only with genera $0$ and $1$. Also this allows us to write down the action of a Hodge field theory.

In Section 7, we prove the main technical result (Main Lemma). Informally, it states that $Q=-G_-\psi$ when we apply these two operators to the correlators of a Hodge field theory. In order to prove it, we look at a small piece (consisting just of one heavy edge and one or two vertices that are attached to it) in one of the graphs of a correlator. Of course, in the correlator we can vary this small piece in an arbitrary way, such that the rest of the graph remains the same. So, when we consider the sum of all these small pieces, it is also a correlator of the Hodge field theory. Thus we reduce the proof to a special case of the whole statement. But since the genus of a vertex is $\leq 1$, it appear now to be a low-genera statement that can be done by a straightforward calculation.

In Section 8, we present the proof of the Main Theorem. Consider a $\psi$-$\kappa$-stratum $\alpha$ whose stable dual graph has $k\geq 1$ edges. There is a universal expression of the integral over $\alpha$ coming from Gromov-Witten theory. It includes $k$ entries of the scalar product restricted to $H_0$. In terms of graphs, it means that we are to introduce new edges with the bivector $[\Pi_0]$ on them, and there are $k$ such edges in our expression. A direct corollary of the Main Lemma is that we can always replace $[\Pi_0]$ by $[Id]-\psi [G_-G_+]-[G_-G_+]\psi$. 

In Sections 8.2, 8.3, and 8.4, we show that when we replace $[\Pi_0]$ by $[Id]-\psi [G_-G_+]-[G_-G_+]\psi$ at all edges corresponding to the scalar product restricted to $H_0$, we obtain a new expression for the integral over $\alpha$ that again contains only heavy edges and empty loops, as any ordinary correlator. The main problem now is to understand the combinatorial coefficient of a graph $\Gamma$ obtained this way. 

Since we have a sum over graphs with heavy edges and empty loops, it is natural to identify again these graphs with the corresponding strata in the moduli space of curves. Then we can calculate the intersection index of the stratum corresponding to a graph $\Gamma$ and the initial class $\alpha$. Roughly speaking, the main thing that we have to do is to decide about each node in $\alpha$ (represented initially by $[\Pi_0]$), whether we have this node in stratum corresponding to $\Gamma$. If yes, then we have an excessive intersection (so, we must put $-\psi$ on one of the half-edges of the corresponding edge), and we keep this edge in $\Gamma$ (so, we replace $[\Pi_0]$ with $-\psi[G_-G_+]$ or $-[G_-G_+]\psi$). If no, then we don't have this edge in $\Gamma$, so we contract $[\Pi_0]$, i.~e., replace it with $[Id]$.

So, the procedure that we used to get rid of the scalar product is the same as the procedure of the intersection of $\alpha$ with strata of the complementary dimension. This means (Section 8.5) that the universal formula coming from Gromov-Witten theory is equivalent to the natural formula for the ``correlator with $\alpha$'' coming from Zwiebach theory. The tautological relation is a sum of classes equal to zero. So, while the universal formula coming from Gromov-Witten theory gives (in the case of a vanishing class) a non-trivial expression in correlators, the natural formula coming from Zwiebach theory gives identically zero. This proves our theorem, see Section 8.6.

%%%%%%%%%%%%%%%%%%%%%%%%%%%%%%%%%%%%%%%%%%%%%%%%%%%%%%%%%%
%%%%%%%%%%%%%%%%%%%%%%%%%%%%%%%%%%%%%%%%%%%%%%%%%%%%%%%%%%
%%%%%%%%%%%%%%%%%%%%%%%%%%%%%%%%%%%%%%%%%%%%%%%%%%%%%%%%%%
%%%%%%%%%%%%%%%%%%%%%%%%%%%%%%%%%%%%%%%%%%%%%%%%%%%%%%%%%%
%%%%%%%%%%%%%%%%%%%%%%%%%%%%%%%%%%%%%%%%%%%%%%%%%%%%%%%%%%

\section{Vanishing of the BV structure}

In this section, we recall several useful lemmas shared in~\cite{t,s}. In
particular, these lemmas give some strong restrictions on graphs that can give a
non-zero contirbution to the correlators defined above.

\subsection{Lemmas}

\begin{lemma}\label{tillmann}~\cite{t,s} The following vectors and bivectors are equal to zero:
\begin{equation}
\inspic{vanish.1}, \quad \inspic{vanish.2}, \quad \inspic{vanish.3}, \quad
\inspic{vanish.4}.
\end{equation}
\end{lemma}

Also let us remind another lemma in~\cite{s} that is very useful in
calculations.

\begin{lemma}\label{g-jump}~\cite{s}
For any vectors $V_0,V_1,\dots,V_k$, $k\geq 2$, 
\begin{align}
\inspic{vanish.5} & = \inspic{vanish.6}+\dots+\inspic{vanish.7}, \\
\inspic{vanish.8} & = \inspic{vanish.9}+\dots+\inspic{vanish.10}.
\end{align}
\end{lemma}

Both lemmas are just simple corollaries of the axioms of cyclic Hodge algebra. 

\subsection{Structure of graphs}\label{action}

Consider a graph studied in Section~\ref{subsec-correlators}. It can have
leaves, empty and heavy loops, and heavy edges. Consider a vertex of such graph.
Let us assume that there are $A$ empty loops, $B$ heavy loops, $C$ heavy edges
going to the other vertices of the graph, and $D$ leaves attached to this
vertex:
\begin{equation}
\inspic{vanish.11}.
\end{equation}
This picture can be considered as an $C+D$ form. Let us denote it by
$\Phi(A,B,C,D)$.

\begin{lemma}\label{lemma-emptyedges}
If $A\geq 2$ and $B+C\geq 1$, then $\Phi(A,B,C,D)=0$.
\end{lemma}

In other words, if there are at least two empty loops at a vertex, then there
should not be any heavy loops or edges attached to this vertex. Otherwise the
contribution of the whole graph vanishes. This implies

\begin{corollary}
In the definition of correlators one should consider only graphs of one of the
following two types:
\begin{enumerate}
\item
One-vertex graphs with no heavy edges (loops).
\item
Arbitrary graphs with at most one empty loop at each vertex.
\end{enumerate}
The contribution of all other graphs vanishes.
\end{corollary}

This corollary dramatically simplifies all our calculations with graphs given
below. Also we can write down now the action of the Hodge field theory. 

Let $F_g^0(v_0,v_1,v_2,\dots)$, $v_i\in H\otimes\C[[\{T_{n,i}\}]]$ be the ``dimension zero'' part of the potential of the Hodge field theory, namely,
\begin{equation}
F_g^0:=\sum_{n}\frac{1}{n!}\sum_{a_1+\dots+a_n=3g-3+n}\langle \tau_{a_1}(v_{a_1})\dots \tau_{a_n}(v_{a_n})\rangle_g.
\end{equation}
The first sum is taken over $n\geq 0$ such that $2g-2+n>0$. So, it is exactly the generating function for the vertices of our graph expressions. Then the action of the Hodge field theory is equal to
\begin{multline}
A(v):=F_0^0(E_0+G_-v, E_1, E_2, \dots)+\hbar F_1^0(E_0+G_-v, E_1, E_2, \dots) \\
+\sum_{g\geq 2}\hbar^g F_g^0(E_0, E_1, E_2, \dots)-\frac{1}{2}\int Qv\cdot G_-v.
\end{multline}
If we put $T_{n,i}=0$ for $n\geq 1$, then we immediately obtain the BCOV-type action discussed in~\cite[Appendix]{bk} and~\cite[Appendix]{ls}. The similar actions were also studied in~\cite{dijk} and~\cite{gersha}.

\subsection{Proof of Lemma~\ref{lemma-emptyedges}}

We consider the form $\Phi(A,B,C,D)$ and we assume that $A\geq 2$.

First, let us study the case when $C\geq 1$. In this case, our $C+D$-form can
be represented as a contraction via the bivector $[Id]$ of two forms,
$\Phi(A-2,B,C-1,D+1)$ and $\Phi(2,0,1,1)$. Let us prove that the last one is
equal to zero. Indeed, this two-form can be represented as
$\tilde\Phi(\cdot,G_+\cdot)$, where the two-form $\tilde\Phi$ is represented by
the picture
\begin{equation}\label{vanishing-bivector}
\inspic{vanish.3}.
\end{equation}

According to Lemma~\ref{tillmann}, $\tilde\Phi=0$. Therefore, $\Phi(2,0,1,1)=0$
and the whole form $\Phi(A,B,C,D)$ is also equal to zero.

Now consider the case when $B\geq 1$. In this case, our $C+D$-form can
be represented as a contraction via the bivector $[Id]$ of two forms,
$\Phi(A-2,B-1,C,D+1)$ and $\Phi(2,1,0,1)$. Let us prove that the last one is
equal to zero. Indeed,
\begin{multline}
\Phi(2,1,0,1)=\inspic{vanish.12}=\inspic{vanish.13}\\
=\frac{1}{2}\inspic{vanish.14}=\frac{1}{2}\inspic{vanish.15}.
\end{multline}
Here the first equality is definition of $\Phi(2,1,0,1)$, the second one is just
an equivalent redrawing, the third equality is application of
Lemma~\ref{g-jump}, the fourth one is again an equivalent redrawing.

The last picture contains the bivector~\eqref{vanishing-bivector} which is equal
to zero according to Lemma~\ref{tillmann}. Therefore, the whole picture is
equal to zero, and $\Phi(2,1,0,1)=0$. So, the whole form $\Phi(A,B,C,D)$ is
equal to zero also in this case. This proves the lemma.

%%%%%%%%%%%%%%%%%%%%%%%%%%%%%%%%%%%%%%%%%%%%%%%%%%%%%%%%%%
%%%%%%%%%%%%%%%%%%%%%%%%%%%%%%%%%%%%%%%%%%%%%%%%%%%%%%%%%%
%%%%%%%%%%%%%%%%%%%%%%%%%%%%%%%%%%%%%%%%%%%%%%%%%%%%%%%%%%
%%%%%%%%%%%%%%%%%%%%%%%%%%%%%%%%%%%%%%%%%%%%%%%%%%%%%%%%%%
%%%%%%%%%%%%%%%%%%%%%%%%%%%%%%%%%%%%%%%%%%%%%%%%%%%%%%%%%%

\section{Main Lemma}

\subsection{Statement}
The main technical tool that we use in the proof of Theorem~\ref{thm-tautolog} 
is the lemma that we prove in this section.

\begin{lemma}[Main Lemma] \label{main-lemma} For any $v_1,\dots,v_n\in H$, $a_1,\dots,a_n\geq 0$,
\begin{multline}\label{main-lemma-statement}
\sum_{i=1}^n\langle \tau_{a_1}(v_1)\dots
\tau_{a_{i-1}}(v_{i-1})
\tau_{a_i}(Q(v_i))
\tau_{a_{i+1}}(v_{i+1})
\dots
\tau_{a_n}(v_n) \rangle_g = \\
-\sum_{i=1}^n\langle \tau_{a_1}(v_1)\dots
\tau_{a_{i-1}}(v_{i-1})
\tau_{a_i+1}(G_-(v_i))
\tau_{a_{i+1}}(v_{i+1})
\dots
\tau_{a_n}(v_n) \rangle_g.
\end{multline}
\end{lemma}

A simple corollary of this lemma is the following:

\begin{lemma}\label{cor-main-l} For any $w\in H$, $v_1,\dots,v_n\in H_0$,
\begin{multline}\label{corollary-lemma-statement}
\langle \tau_{a_0}(Qw)\tau_{a_1}(v_1)\dots \tau_{a_n}(v_n) \rangle_g\\
+\langle \tau_{a_0+1}(G_-w)\tau_{a_1}(v_1)\dots \tau_{a_n}(v_n) \rangle_g = 0.
\end{multline}
\end{lemma}

In other words, we can state informally that $Q+\psi G_-=0$.

\subsection{Special cases}

The proof of the lemma can be reduced to a small number of special cases. 
We consider correlators whose graphs have only one heavy edge.

\subsubsection{} The first case is the following:
Let $\sum_{i=1}^n a_i = n-4$. We prove that for any $v_1,\dots,v_n\in H$,
\begin{multline}\label{special-case-1}
\sum_{i=1}^n\langle \tau_{a_1}(v_1)\dots
\tau_{a_{i-1}}(v_{i-1})
\tau_{a_i}(Q(v_i))
\tau_{a_{i+1}}(v_{i+1})
\dots
\tau_{a_n}(v_n) \rangle_0 = \\
-\sum_{i=1}^n\langle \tau_{a_1}(v_1)\dots
\tau_{a_{i-1}}(v_{i-1})
\tau_{a_i+1}(G_-(v_i))
\tau_{a_{i+1}}(v_{i+1})
\dots
\tau_{a_n}(v_n) \rangle_0.
\end{multline}

First, we see that according to the definition of the correlator, the 
left hand side of Equation~\eqref{special-case-1} is the sum over graphs 
with two vertices and with $[G_-G_+]$ on the unique edge that connects the 
vertices. For each $I\sqcup J=\{1,\dots,n\}$ we can consider
the corresponding distribution of leaves between the vertices (to be precise, 
let us assume that $1\in I$). Then the coefficient of such graph is 
$\langle \tau_0\prod_{i\in I}\tau_{a_i}\rangle_0\langle \tau_0\prod_{j\in J}\tau_{a_j}\rangle_0$, 
and we take the sum over all possible positions of $Q$ at the leaves.

Using the Leibniz rule for $Q$ and the property that $[Q,G_-G_+]=-G_-$, we see that this sum is
equal to the sum over graphs with two vertices and with $[-G_-]$ on the unique edge that 
connects the vertices. For each $I\sqcup J=\{1,\dots,n\}$, $|I|,|J|\geq 2$, we consider
the corresponding distribution of leaves between the vertices. Then the coefficient of 
such graph is still 
$\langle \tau_0\prod_{i\in I}\tau_{a_i}\rangle_0\langle \tau_0\prod_{j\in J}\tau_{a_j}\rangle_0$, 
and the underlying tensor expression can be written (after we multiply the whole sum by $-1$) as
\begin{equation}
\left(v_1,\prod_{i\in I\setminus\{1\}}v_i\cdot G_-(\prod_{j\in J}v_j)\right).
\end{equation}

Let us recall that the $7$-term relation for $G_-$ implies that
\begin{align}\label{7-term-corr}
G_-(\prod_{j\in J}v_j)=&\sum_{i,j\in J,\ i<j}
G_-(v_iv_j)\prod_{k\in J \setminus\{i,j\}}v_k \\
&-(|J|-2)\sum_{j\in J}
G_-(v_j)\prod_{i\in J \setminus\{i,j\}}v_i.\notag
\end{align}
Using this, we can rewrite the whole sum over graphs as
\begin{multline}\label{rewrite-1}
\sum_{1<i<j}\left[\left(v_1,\prod_{k\not=1,i,j}v_k\cdot G_-(v_iv_j)\right)\cdot\right. \\
\left.\sum_{I'\sqcup J'
\sqcup\{i,j\}=\{2,\dots,n\}}
\langle
\tau_{a_1}
\tau_0
\prod_{k\in I'}
\tau_{a_k}
\rangle_0
\langle
\tau_{a_i}
\tau_{a_j}
\tau_0
\prod_{k\in J'}
\tau_{a_k}
\rangle_0\right]\\
-\sum_{i\not=1}\left[\left(v_1,\prod_{j\not=1,i}v_j\cdot G_-(v_i)\right)\cdot\right. \\
\left.\sum_{I'\sqcup J'
\sqcup\{i\}=\{2,\dots,n\}}
\left(|J'|-1\right)\langle
\tau_{a_1}
\tau_0
\prod_{j\in I'}
\tau_{a_j}
\rangle_0
\langle
\tau_{a_i}
\tau_0
\prod_{j\in J'}
\tau_{a_j}
\rangle_0\right].
\end{multline}
Using that
\begin{multline}
\sum_{I'\sqcup J'
\sqcup\{i,j\}=\{2,\dots,n\}}
\langle
\tau_{a_1}
\tau_0
\prod_{k\in I'}
\tau_{a_k}
\rangle_0
\langle
\tau_{a_i}
\tau_{a_j}
\tau_0
\prod_{k\in J'}
\tau_{a_k}
\rangle_0
=\\
\langle
\tau_{a_1+1}
\prod_{k\not=1}
\tau_{a_k}
\rangle_0,
\end{multline}
Equation~\eqref{7-term-corr}, and the fact that
\begin{multline}
(n-3)\langle
\tau_{a_1+1}
\prod_{j\not=1}
\tau_{a_j}
\rangle_0
\\
-\sum_{I'\sqcup J'
\sqcup\{i\}=\{2,\dots,n\}}
\left(|J'|-1\right)\langle
\tau_{a_1}
\tau_0
\prod_{j\in I'}
\tau_{a_j}
\rangle_0
\langle
\tau_{a_i}
\tau_0
\prod_{j\in J'}
\tau_{a_j}
\rangle_0
\\
=
\langle
\tau_{a_i+1}
\prod_{j\not=i}
\tau_{a_j}
\rangle_0,
\end{multline}
we can rewrite Expression~\eqref{rewrite-1} as
\begin{equation}
\sum_{i=1}^n\left(G_-(v_i),\prod_{j\not=i}v_j\right)\langle\tau_{a_i+1}
\prod_{j\not=i}\tau_{a_j}\rangle_0.
\end{equation}
The last formula coincides by definition with the right hand side of 
Equation~\eqref{special-case-1} multiplied by $-1$. This proves the first 
special case.

\subsubsection{}
The second case is in genus $1$. Let $\sum_{i=1}^n a_i=n-1$. We prove
that for any $v_1,\dots,v_n\in H$,
\begin{multline}\label{special-case-2}
\sum_{i=1}^n\langle \tau_{a_1}(v_1)\dots
\tau_{a_{i-1}}(v_{i-1})
\tau_{a_i}(Q(v_i))
\tau_{a_{i+1}}(v_{i+1})
\dots
\tau_{a_n}(v_n) \rangle_1 = \\
-\sum_{i=1}^n\langle \tau_{a_1}(v_1)\dots
\tau_{a_{i-1}}(v_{i-1})
\tau_{a_i+1}(G_-(v_i))
\tau_{a_{i+1}}(v_{i+1})
\dots
\tau_{a_n}(v_n) \rangle_1.
\end{multline}

According to the definition of the correlator, the left hand side of 
Equation~\eqref{special-case-2} is the sum over graphs of two possible types.
The first type include graphs with two vertices and two edges. The first 
edge is heavy and connects the vertices; the second edge is 
an empty loop attached to the first vertex. For each $I\sqcup J=\{1,\dots,n\}$, $|J|\geq 2$, 
we can consider
the corresponding distribution of leaves between the vertices (we assume that 
leaves with indices in $I$ are at the first edge). Then the coefficient of such 
graph is $\langle \tau_0\prod_{i\in I}\tau_{a_i}\rangle_1\langle \tau_0\prod_{j\in J}\tau_{a_j}\rangle_0$.
The second type include graphs with one vertex and one heavy loop. 
All leaves are attached to this vertex, and the coefficient of such graph is
$\langle \tau_0^2\prod_{i=1}^n\tau_{a_i}\rangle_0$.
For both types of graphs, we take the sum over all possible positions of $Q$ at 
the leaves.

Using the Leibniz rule for $Q$ and the property that $[Q,G_-G_+]=-G_-$, we get 
the same graphs as before, but there is no $Q$, and instead of $[G_-G_+]$ we
have $-[G_-]$ on the corresponding edge.
Using Lemma~\ref{g-jump}, we move $-G_-$ in graphs of the first type 
to the leaves marked by indices in $J$. Using the $1/12$-axiom and Lemma~\ref{g-jump}, 
we move $-G_-$ in graphs of the second type to all leaves.

This way we get graphs of the same type in both cases. We get graphs with one vertex, 
one empty loop attached to it, all leaves are also attached to this vertex, and there 
is $-G_-$ on one of the leaves. One can easily check that the coefficient of the graphs 
with $-G_-$ at the $i$-th leaf is equal to
\begin{multline}
\sum_{I\sqcup J \sqcup \{i\}=\{1,\dots,n\}}
\langle \tau_0\prod_{k\in I}\tau_{a_k}\rangle_1\langle \tau_0\tau_{a_i}\prod_{k\in J}\tau_{a_k}\rangle_0
+ \frac{1}{24}\langle \tau_0^2\prod_{k=1}^n\tau_{a_k}\rangle_0 \\
= \langle \tau_{a_i+1}\prod_{k\not= i}\tau_{a_k}\rangle_1
\end{multline}
It is exactly the unique graph contributing to the $i$-th summand of the right hand side of Equation~\eqref{special-case-2},
and the coefficient is right. This proves that special case.

\subsubsection{}
Consider $g\geq 2$. Let $\sum_{i=1}^n a_i=3g+n-4$. In this case, the statement that
for any $v_1,\dots,v_n\in H$,
\begin{multline}\label{special-case-3}
\sum_{i=1}^n\langle \tau_{a_1}(v_1)\dots
\tau_{a_{i-1}}(v_{i-1})
\tau_{a_i}(Q(v_i))
\tau_{a_{i+1}}(v_{i+1})
\dots
\tau_{a_n}(v_n) \rangle_g = \\
-\sum_{i=1}^n\langle \tau_{a_1}(v_1)\dots
\tau_{a_{i-1}}(v_{i-1})
\tau_{a_i+1}(G_-(v_i))
\tau_{a_{i+1}}(v_{i+1})
\dots
\tau_{a_n}(v_n) \rangle_g,
\end{multline}
is immediately reduced to $0=0$; it is a simple corollary of Lemma~\ref{tillmann}.

\subsection{Proof of Main Lemma}

Consider the left hand side of Equation~\eqref{main-lemma-statement}.
As usual, using the Leibniz rule for $Q$ and the property that $[Q,G_-G_+]=-G_-$,
we can remove all $Q$, but then we must change one of $[G_-G_+]$ on edges to $-[G_-]$.
Let us cut out the peaces of graphs that includes this edges with $-[G_-]$, all empty 
loops, leaves and  halves of heavy edges attached to the ends of this special 
edge.

Since we consider the sum over all possible graphs contributing to correlators, these 
small pieces can be gathered into groups according to the type of the rest of the 
initial graph. Each group forms exactly one of the special cases studied above. So, we 
know that $-G_-$ should jump either to one of the leaves or to one of the heavy edges 
attached to the ends of its edge. In the 
first case, we get exactly the graphs in the right hand side of 
Equation~\eqref{main-lemma-statement}; in the second case, we get zero. One can 
easily check that we get the right coefficients for the graphs in the right hand 
side of Equation~\eqref{main-lemma-statement}. This proves the lemma.

%%%%%%%%%%%%%%%%%%%%%%%%%%%%%%%%%%%%%%%%%%%%%%%%%%%%%%%%%%
%%%%%%%%%%%%%%%%%%%%%%%%%%%%%%%%%%%%%%%%%%%%%%%%%%%%%%%%%%
%%%%%%%%%%%%%%%%%%%%%%%%%%%%%%%%%%%%%%%%%%%%%%%%%%%%%%%%%%
%%%%%%%%%%%%%%%%%%%%%%%%%%%%%%%%%%%%%%%%%%%%%%%%%%%%%%%%%%
%%%%%%%%%%%%%%%%%%%%%%%%%%%%%%%%%%%%%%%%%%%%%%%%%%%%%%%%%%

\section{Proof of Theorem~\ref{thm-tautolog}}

\subsection{Equivalence of expression in graphs}

Consider the expression in correlators corresponding to
a $\psi$-$\kappa$-stratum as it is described in Section~\ref{integrals-strata}. 
To each vertex of the corresponding stable dual graph we assign the sum of
graphs that forms correlator in the sense of Section~\ref{subsec-correlators}.
The leaves of these graphs corresponding to the edges of the stable dual
graph (nodes) are connected in these pictures by edges with $[\Pi_0]$ (the restriction 
of the scalar product to $H_0$). We call the edges with $[\Pi_0]$ ``white
edges'' and mark them in pictures by thick white points, see~\eqref{edges}.

The axioms of cyclic Hodge algebra imply a system of linear equations for the graphs of
this type. In particular, it has appeared that playing with this linear equations we 
can always get rid of white edges in the sum of pictures corresponding to a
stable dual graph, see~\cite{ls,s,ss}. However, previously it was just an experimental
fact. Now we can show how it works in general.

The numerous examples of the correspondence between stable dual graphs and
graphs expressions in cyclic Hodge algebras and also of the linear relations implied by
the axioms of cyclic Hodge algebra are given in~\cite{ls,s,ss}.

Below, we explain how one can represent the expression in correlators
corresponding to a $\psi$-$\kappa$-stratum in terms of graphs with only empty
and heavy edges and with no white edges. The unique tool that we need is
Lemmas~\ref{main-lemma} and~\ref{cor-main-l} proved above.

\subsection{The simplest example}

Consider a stable dual graph with two vertices and one edge connecting them:
\begin{equation}
\inspic{pics.1}
\end{equation}
The corresponding expression in correlators is
\begin{multline}
\langle\tau_{a_0}(e_{j_1})\prod_{i=1}^{n_1}\tau_{a_i}(e_\bullet)
\prod_{i=1}^{l_1}\tau_{u_i}(e_1)\rangle_{g_1}\cdot
\\
\eta^{j_1j_2}\cdot
\langle\tau_{b_0}(e_{j_2})\prod_{i=1}^{n_2}\tau_{b_i}(e_\bullet)
\prod_{i=1}^{l_2}\tau_{v_i}(e_1)\rangle_{g_1}
\end{multline}
(here we denote by $e_\bullet$ an arbitrary choice of $e_i\in H_0$).
It is convenient for us to rewrite this expression as
\begin{multline}
\langle\tau_{a_0}(x_{\alpha_1})\prod_{i=1}^{n_1}\tau_{a_i}(e_\bullet)
\prod_{i=1}^{l_1}\tau_{u_i}(e_1)\rangle_{g_1}\cdot
\\
[\Pi_0]^{\alpha_1\alpha_2}\cdot
\langle\tau_{b_0}(x_{\alpha_2})\prod_{i=1}^{n_2}\tau_{b_i}(e_\bullet)
\prod_{i=1}^{l_2}\tau_{v_i}(e_1)\rangle_{g_1},
\end{multline}
where $\{x_\alpha\}$ is the basis of the whole $H$. Using the fact that
$\Pi_0=Id-QG_+-G_+Q$ and applying Lemma~\ref{cor-main-l}, we
obtain
\begin{multline}\label{1edge}
\langle\tau_{a_0}(x_{\alpha_1})\prod_{i=1}^{n_1}\tau_{a_i}(e_\bullet)
\prod_{i=1}^{l_1}\tau_{u_i}(e_1)\rangle_{g_1}\cdot
\\ [\Pi_0]^{\alpha_1\alpha_2}\cdot
\langle\tau_{b_0}(x_{\alpha_2})\prod_{i=1}^{n_2}\tau_{b_i}(e_\bullet)
\prod_{i=1}^{l_2}\tau_{v_i}(e_1)\rangle_{g_1} \\
=
\langle\tau_{a_0}(x_{\alpha_1})\prod_{i=1}^{n_1}\tau_{a_i}(e_\bullet)
\prod_{i=1}^{l_1}\tau_{u_i}(e_1)\rangle_{g_1}\cdot
\\
[Id]^{\alpha_1\alpha_2}\cdot
\langle\tau_{b_0}(x_{\alpha_2})\prod_{i=1}^{n_2}\tau_{b_i}(e_\bullet)
\prod_{i=1}^{l_2}\tau_{v_i}(e_1)\rangle_{g_1} \\
-
\langle\tau_{a_0+1}(x_{\alpha_1})\prod_{i=1}^{n_1}\tau_{a_i}(e_\bullet)
\prod_{i=1}^{l_1}\tau_{u_i}(e_1)\rangle_{g_1}\cdot
\\
[G_-G_+]^{\alpha_1\alpha_2}\cdot
\langle\tau_{b_0}(x_{\alpha_2})\prod_{i=1}^{n_2}\tau_{b_i}(e_\bullet)
\prod_{i=1}^{l_2}\tau_{v_i}(e_1)\rangle_{g_1} \\
-
\langle\tau_{a_0}(x_{\alpha_1})\prod_{i=1}^{n_1}\tau_{a_i}(e_\bullet)
\prod_{i=1}^{l_1}\tau_{u_i}(e_1)\rangle_{g_1}\cdot
\\
[G_-G_+]^{\alpha_1\alpha_2}\cdot
\langle\tau_{b_0+1}(x_{\alpha_2})\prod_{i=1}^{n_2}\tau_{b_i}(e_\bullet)
\prod_{i=1}^{l_2}\tau_{v_i}(e_1)\rangle_{g_1}
\end{multline}
In all three summands of the right hand side we still have two correlators,
whose leaves corresponding to the nodes are connected by some special edges. 
But now the connecting edge is either marked by $[Id]$ (an empty edge) or by
$[G_-G_+]$ (an ordinary heavy edge). So, this way we get rid of the white edge
in this case.

Informally, in terms of pictures, we can describe Equation~\eqref{1edge} as
\begin{align}
\inspic{algpic.1}& =
\inspic{algpic.2} 
-\inspic{algpic.3}
\\ &
-\inspic{algpic.4}.\notag
\end{align}
When we put $\psi$, we mean that we add one more $\psi$-class at the node at the
corresponding branch of the curve. Dashed circles denote correlators.

\subsection{Example with two nodes} Now we consider an example of stratum, 
whose generic point is represented by a three-component curve. Again, we 
allow arbitrary $\psi$-classes at marked points and two branches at nodes. 

We perform the same calculation as above, but now we explain it in terms of informal 
pictures from the very beginning. So, the first step is the same as above:
\begin{align}
\inspic{algpic.5}&=\inspic{algpic.6}\\
& -\inspic{algpic.7} \notag \\
& -\inspic{algpic.8}. \notag
\end{align}
Then we apply Lemma~\ref{main-lemma} to each of the summands 
in the right hand side:
\begin{align}
\inspic{algpic.6}&=\inspic{algpic.9}\\
&-\inspic{algpic.10} \notag \\
&-\inspic{algpic.11} \notag \\
&-\inspic{algpic.12} \notag \\
&-\inspic{algpic.13}, \notag 
\end{align}
\begin{align}
-\inspic{algpic.7}& =-\inspic{algpic.14}\\
&+\inspic{algpic.15} \notag \\
&+\inspic{algpic.16} \notag \\
&+\inspic{algpic.12},  \notag 
\end{align}
and
\begin{align}
-\inspic{algpic.8} & =-\inspic{algpic.17}\\
&+\inspic{algpic.18}  \notag \\
&+\inspic{algpic.19}  \notag \\
&+\inspic{algpic.13},  \notag 
\end{align}
We take the sum of these three expressions, 
and we see that all pictures where we have 
edges with $[G_-]$ and $[G_+]$ are cancelled. 
So, we get an expression for the sum of graphs 
representing the initial stratum in terms of 
graphs with only empty and heavy edges. 

\subsection{General case}
The general argument is exactly the same as in the second example.
In fact, this gives a procedure how to write an expression in graphs with only
empty and heavy edges (and no white edges) starting from a stable dual graph. 
Let us describe this
procedure.
 
Take a stable dual graph corresponding to a $\psi$-$\kappa$-stratum of dimension
$k$ in $\oM_{g,n}$. First, we are to decorate it a little bit. 
For each edge, we either leave it untouched, or substitute it with an arrow 
(in two possible ways). At the pointing end of the arrow, we increase the 
number of $\psi$-classes by $1$. Each of these graphs we weight with the 
inversed order of its automorphism group (automorphisms must preserve 
all decorations) multiplied by $(-1)^{arr}$, where $arr$ is the number 
of arrows. 

Consider a decorated dual graph. To each its vertex we associate the
corresponding correlator of cyclic Hodge algebra (we add new leaves in order to represent 
$\kappa$-classes). Then we connect the leaves corresponding
to the nodes either by empty edges (if the corresponding edge of the decorated
graph is untouched) or by heavy edges (if the corresponding edge of the dual
graph is decorated by an arrow).

It is obvious that the number of heavy edges in the final graphs is equal to $k$. 

\subsection{Coefficients}
We can simplify the resulting graphs obtained in the previous subsection. First,
we can contract empty edges (as much as it is possible; it is forbidden to
contract loops). Second, we can remove
leaves added for the needs of $\kappa$-classes. Indeed, each such leaf is
equipped with a unit of $H$, so it doesn't affect the contraction of tensors
corresponding to a graph. Moreover, when we remove all leaves corresponding to
$\kappa$-classes, we still have graph with at least trivalent vertices.
Otherwise, this graph is equal to zero, c.~f. arguments in the proofs of
string and dilaton equations. 

So, we obtain final graphs that have the same number of heavy edges as the dimension
of the initial $\psi$-$\kappa$-stratum, the same number of leaves as the initial
dual graph, and some number of empty loops, at most one at each vertex. The
exceptional case is when $k=0$; in this case we obtain only one graph, with 
one vertex, $n$ leaves, and $g$ empty loops.

In the first case, let us turn a graph like this into a stable dual graph. Just replace
its vertices with no empty loops by vertices of genus zero, vertices with
empty loops by vertices of genus one, heavy edges are edges, and leaves are
leaves. There are no $\psi$- or $\kappa$- classes. It is obvious that the
codimension of the stratum corresponding to this dual graph is $k$. Indeed, in
this case it is just the number of nodes. 

So, to each $\psi$-$\kappa$-stratum $X$ of dimension $k>0$ we associate a linear
combination $\sum_i c_i Y_i$ of strata of codimension $k$ with no $\psi$- or
$\kappa$-classes, whose curves have irreducible components of genus $0$ and $1$
only. 

\begin{proposition}\label{prop1}
We have $c_i=X\cdot Y_i$.
\end{proposition}

\begin{proof}
We prove it two steps. First, consider a one-vertex stable dual graph with no
edges (just a correlator). In this case, the intersection number $X\cdot Y_i$ is 
just by definition $c_i=V(\Gamma_i)P(\Gamma_i)$, where $\Gamma_i$ is the
cyclic Hodge algebra graph that turns into 
$Y_i$ via the procedure described above.

Then, consider a stable dual graph with one edge. It is the intersection of the
one-vertex stable dual graph with an irreducible component of the boundary. For
a given $Y_i$, this component of the boundary either intersects it transversaly,
or we have an excessive intersection. In the first case, the corresponding node
is not represented in $Y_i$. This means that in $\Gamma_i$ it should be an
empty edge. In the second case, this node is one of the nodes of $Y_i$, so it
should be a heavy edge of $\Gamma_i$. Also, it is an excessive intersection, so
we are to add the sum of $\psi$-classes with the negative sign at the marked
points (half-edges) corresponding to the node, see~\cite[Appendix]{grpa}.

Exactly the same argument works for an arbitrary number of nodes, we just 
extend it by induction.
\end{proof}

In the case of $k=0$, we get just one final graph with coefficient $c$.

\begin{proposition}\label{prop2}
If $k=0$, the coefficient of the final graph is equal to the number of points in
the initial $\psi$-$\kappa$-stratum.
\end{proposition}

\begin{proof}
If $k=0$, this means that each of the vertices of the initial stable dual graph
also has dimension $0$, and the corresponding correlator of cyclic Hodge algebra is
represented by one one-vertex graph with no heavy edges. Also this means that 
each edge of the initial stable dual graph is replaced in the algorithm above 
by an empty edge. So, we can think that we just work with the correlators of the
Gromov-Witten theory of the point. In this case Proposition becomes obvious.
\end{proof}

Now we deduce Theorem~\ref{thm-tautolog} from these propositions.

\subsection{Proof of Theorem~\ref{thm-tautolog}}

Consider the system of subalgebras 
\begin{equation}
RH_1^*(\oM_{g,n})\subset RH^*(\oM_{g,n})
\end{equation}
of the cohomological tautological algebras of $\oM_{g,n}$
generated by strata with no $\psi$- or $\kappa$-classes and with irreducible
curves of genus $0$ and $1$ only. 

Let $L$ be a linear combination of $\psi$-$\kappa$-strata of dimension $k$ in
$\oM_{g,n}$. Then the expression in correlators of cyclic Hodge algebras corresponding to 
$L$ is equaivalent to a sum of some graphs with coefficients equal to the 
intersection of $L$ with classes in $RH_1^k(\oM_{g,n})$.

So, if the class of $L$ is equal to zero, then the corresponding equation (and
also the whole system of equations that we described in Section~\ref{rel-for-corr}) 
for correlators of cyclic Hodge algebra is valid. Theorem is proved.

\subsubsection{Remark}\label{last-remark}
Evidently, $RH_1^*(\oM_{g,n})$ is a
module over $RH^*(\oM_{g,n})$. Also it is obvious that $RH_1^*(\oM_{g,n})$ is
closed under pull-backs and push-forwards via the forgetful morphisms.
This explains why it was enough to make only one check in the simplest case in
order to get the system of equations in~\cite{ls,ss} (cf.~an argument in the 
last section in~\cite{ls}).

\subsection{An interpretation of Propositions~\ref{prop1} and~\ref{prop2}}

From the point of view of the theory of Zwiebach invariants, both propositions
look very natural. Indeed, we try to give a graph
expression for the integral of an induced Gromov-Witten form
multiplied by a tautological class $X$. Since we know that we are able to
integrate only degree zero parts of induced Gromov-Witten invariants, we should just take 
the sum over all graphs that correspond to the strata of complimentary dimension
in $RH_1^*(\oM_{g,n})$. The coefficients are to be the intersection numbers of
these strata with $X$.

On the other hand, we know that in any Gromov-Witten theory it is enough to fix
the integrals of Gromov-Witten invariants multiplied by $\psi$-classes. Then the
integrals of Gromov-Witten invariants multiplied by arbitrary tautological
classes are expressed by universal formulas. We can try to use these universal
formulas also in Hodge field theory. They are exactly our expressions with white edges.

So, we have two different natural ways to express in terms of graphs the 
integrals of induced Gromov-Witten invariants multiplied by 
tautological classes. Propositions~\ref{prop1} and~\ref{prop2} state that these
two different expressions coinside.

%%%%%%%%%%%%%%%%%%%%%%%%%%%%%%%%%%%%%%%%%%%%%%%%%%%%%%%%%%
%%%%%%%%%%%%%%%%%%%%%%%%%%%%%%%%%%%%%%%%%%%%%%%%%%%%%%%%%%
%%%%%%%%%%%%%%%%%%%%%%%%%%%%%%%%%%%%%%%%%%%%%%%%%%%%%%%%%%
%%%%%%%%%%%%%%%%%%%%%%%%%%%%%%%%%%%%%%%%%%%%%%%%%%%%%%%%%%
%%%%%%%%%%%%%%%%%%%%%%%%%%%%%%%%%%%%%%%%%%%%%%%%%%%%%%%%%%


\begin{thebibliography}{00}

\bibitem{bk} S. Barannikov, M. Kontsevich, Frobenius manifolds and formality of 
Lie algebras of polyvector fields, Internat. Math. Res. Notices \textbf{1998}, no. 4, 201--215.

\bibitem{bp} P.~Belorousski, R.~Pandharipande, A descendent relation in genus 2,
Ann. Scuola Norm. Sup. Pisa Cl. Sci. (4) \textbf{29} (2000), no. 1, 171--191.

\bibitem{bcov} M. Bershadsky, S. Cecotti, H. Ooguri, C. Vafa, Kodaira-Spencer theory of gravity and exact 
results for quantum string amplitudes, Comm. Math. Phys. \textbf{165} (1994), no. 2, 311--427.

\bibitem{Chen-Li-Liu} Lin Chen, Yi Li, Kefeng Liu, Localization, Hurwitz Numbers and the Witten
Conjecture, arXiv: math.AG/0609263.

\bibitem{dijk} R.~Dijkgraaf, Chiral deformations of conformal field theories, Nuclear Phys.~B \textbf{493} (1997), no.~3, 588--612.

\bibitem{fsz} C.~Faber, S.~Shadrin, D.~Zvonkine, Tautological relations and the 
$r$-spin Witten conjecture, arXiv: math.AG/0612510.

\bibitem{gersha} A.~Gerasimov, S.~Shatashvili, Towards integrability of topological strings I: 
three-forms on Calabi-Yau manifolds, J. High Energy Phys. \textbf{2004}, no. 11, 074.

\bibitem{get} E. Getzler, Batalin-Vilkovisky algebras and two-dimensional 
topological field theories, Comm. Math. Phys. \textbf{159} (1994), 265--285.

\bibitem{g2} E.~Getzler, Intersection theory on $\oM_{1,4}$ and 
elliptic Gromov-Witten invariants.  J. Amer. Math. Soc.  \textbf{10}  (1997),  no. 4, 973--998.

\bibitem{g} E.~Getzler, Topological recursion relations in genus $2$, 
Integrable Systems and Algebraic Geometry (Kobe/Kyoto, 1997), World Scientific, 
River Edge, NJ, 1998, pp. 73--106.

\bibitem{grpa} T. Graber, R. Pandharipande, Constructions of non-tautological classes on 
moduli spaces of curves, Michigan Math. J. 51 (2003), no. 1, 93--109.

\bibitem{hamo} J.~Harris, I.~Morrison, Moduli of curves. Graduate Texts in Mathematics, 
187. Springer-Verlag, New York, 1998.

\bibitem{Kazarian-Lando} M.~Kazarian, S.~Lando, An algebro-geometric proof of Witten's 
conjecture, arXiv: math.AG/0601760.

\bibitem{Kim-Liu} Y.-S.~Kim, K.~Liu, A simple proof of Witten conjecture through 
localization, arXiv: math.AG/0508384.

\bibitem{ksv} T.~Kimura, J.~Stasheff, A.~Voronov,  On operad structures of 
moduli spaces and string theory.  Comm. Math. Phys.  \textbf{171}  (1995),  no. 1, 1--25.

\bibitem{kl} T.~Kimura, X.~Liu, A genus-3 topological recursion relation. 
Comm. Math. Phys. \textbf{262} (2006), no. 3, 645--661.

\bibitem{Kontsevich} M.~Kontsevich, Intersection theory on the moduli space of curves 
and the matrix Airy function, Comm. Math. Phys. \textbf{147} (1992), no. 1, 1--23.

\bibitem{km} M.~Kontsevich, Y.~Manin, Gromov-Witten classes, quantum cohomology, and 
enumerative geometry. Comm. Math. Phys. \textbf{164}  (1994),  no. 3, 525--562.

\bibitem{los} A. Losev, Hodge strings and elements of K. Saito's theory of primitive form, Topological field theory, 
primitive forms and related topics 
(Kyoto, 1996), 305--335, Progr. Math., 160, Birkhaeuser Boston, Boston, MA, 1998.

\bibitem{lm1} A. Losev, Y. Manin, New moduli spaces of pointed curves and pencils of 
flat connections, Michigan Math. J. \textbf{48} (2000), 443--472.

\bibitem{lm2} A. Losev, Yu. Manin, Extended modular operad, Frobenius manifolds, 
181--211, Aspects Math., E36, Vieweg, Wiesbaden, 2004.

\bibitem{ls} A.~Losev, S.~Shadrin, From Zwiebach invariants to Getzler relation, 
Comm. Math. Phys. \textbf{271}, no. 3, 649--679.

\bibitem{man3c} Yu. I. Manin, Three constructions of Frobenius manifolds: a comparative study, 
Asian J. Math. \textbf{3} (1999), no. 1, 179--220.

\bibitem{man} Yu.~Manin, Frobenius manifolds, quantum cohomology, and moduli spaces. 
American Mathematical Society Colloquium Publications, 47. American Mathematical Society, 
Providence, RI, 1999.

\bibitem{mer} S. A. Merkulov, Formality of canonical symplectic complexes and Frobenius manifolds, 
Internat. Math. Res. Notices \textbf{1998}, no. 14, 727--733.

\bibitem{mikha} G. Mikhalkin, Enumerative tropical algebraic geometry in $\mathbb{R}^2$,
J. Amer. Math. Soc. \textbf{18} (2005), no. 2, 313--377. 

\bibitem{Mirzakhani} M.~Mirzakhani, Weil-Petersson volumes and intersection theory on 
the moduli space of curves, J. Amer. Math. Soc. \textbf{20} (2007), no. 1, 1--23. 

\bibitem{mnev} P. Mnev, Notes on simplicial BF theory, arXiv: hep-th/0610326.

\bibitem{Okounkov-Pandharipande} A.~Okounkov, R.~Pandharipande, Gromov-Witten theory, 
Hurwitz numbers, and matrix models, I, arXiv: math.AG/0101147.

\bibitem{schaetz} F. Schaetz, BVF-complex and higher homotopy structures, arXiv:
math.QA/0611912.

\bibitem{s} S.~Shadrin, A definition of descendants at one point in graph calculus, 
arXiv: math.QA/0507106.

\bibitem{ss} S.~Shadrin, I.~Shneiberg, Belorousski-Pandharipande relation in dGBV algebras,
J. Geom. Phys. \textbf{57} (2007), no. 2, 597--615.

\bibitem{shn} I.~Shneiberg, Topological recursion relations in $\oM_{2,2}$, to appear in Funct. Anal. Appl. (2007).

\bibitem{t} U.~Tillmann, Vanishing of the Batalin-Vilkovisky algebra structure for TCFTs, 
Comm. Math. Phys. \textbf{205} (1999), no. 2, 283-286.

\bibitem{Witten} E. Witten, Two dimensional gravity and intersection theory on moduli space. 
Surveys in Differential Geometry, vol.~1 (1991), 243--310.

\bibitem{wit-na} E. Witten, Chern-Simons gauge theory as a string theory, The Floer memorial volume,
 637--678, Progr. Math.~\textbf{133}, Birkhaeuser, Basel, 1995.

\bibitem{zwi} B. Zwiebach, Closed string field theory: quantum action and the Batalin-Vilkovisky 
master equation, Nuclear Phys.~B \textbf{390} (1993), no. 1, 33--152.

\end{thebibliography}
\end{document}